\documentclass[10pt,twoside]{amsart} 
\usepackage{amsmath, amsthm, amscd, amsfonts, amssymb, graphicx, color}
\usepackage[bookmarksnumbered, colorlinks, plainpages]{hyperref}

\newtheorem{theorem}{Theorem}[section]

\newtheorem{proposition}[theorem]{Proposition}
\newtheorem{corollary}[theorem]{Corollary}

\theoremstyle{definition}

\newtheorem{example}[theorem]{Example}

\newtheorem{remark}[theorem]{Remark}

\numberwithin{equation}{section}

\def\<{\langle}
\def\>{\rangle}

\long\def\alert#1{\smallskip{\hskip\parindent\vrule%
\vbox{\advance\hsize-2\parindent\hrule\smallskip\parindent.4\parindent%
\narrower\noindent#1\smallskip\hrule}\vrule\hfill}\smallskip}

\begin{document}

\title[\ci\ ON $\phi$-$n$-ABSORBING PRIMARY IDEALS OF COMMUTATIVE RINGS]{ON $\phi$-$n$-ABSORBING PRIMARY IDEALS OF COMMUTATIVE RINGS}

\author{Hojjat Mostafanasab 
and Ahmad Yousefian Darani}
\address{Department of Mathematics and Applications, University of Mohaghegh Ardabili, P. O. Box 179, Ardabil, Iran} \email{
h.mostafanasab@uma.ac.ir and h.mostafanasab@gmail.com}

\address{Department of Mathematics and Applications, University of Mohaghegh Ardabili, P. O. Box 179, Ardabil, Iran} \email{
yousefian@uma.ac.ir and youseffian@gmail.com}

\thanks{{\scriptsize
\hskip -0.4 true cm 2010 {\it Mathematics Subject Classification}.
Primary 13A15; Secondary 13F05, 13G05.
}}

\begin{abstract}
All rings are commutative with $1\neq0$ and $n$ is a positive integer. Let $\phi:\mathfrak{J}(R)\rightarrow \mathfrak{J}(R)\cup\{\emptyset\}$  be a function
where $\mathfrak{J}(R)$ denotes the set of all ideals of $R$. We say that a proper ideal $I$ of $R$ is  
$\phi$-$n$-absorbing primary if whenever $a_1,a_2,\dots,a_{n+1}\in R$ and $a_1a_2\cdots a_{n+1}\in I\backslash\phi(I)$, 
either $a_1a_2\cdots a_n\in I$ or the product of $a_{n+1}$ with $(n-1)$ of $a_1,\dots,a_n$ is in $\sqrt{I}$.
The aim of this paper is to investigate the concept of $\phi$-$n$-absorbing primary ideals.
\end{abstract}

\maketitle

\section{Introduction}
Throughout this paper $R$ will be a commutative ring with a nonzero identity. 
In \cite{AS}, Anderson and Smith called a proper ideal $I$ of a commutative ring $R$ 
to be {\it weakly prime} if whenever $a,b\in R$ and 
$0\neq ab\in I$, either $a\in I$ or $b\in I$.
In \cite{BS}, Bhatwadekar and Sharma defined a proper ideal $I$ of an integral domain $R$ to be {\it almost prime (resp. $m$-almost prime)} if for $a, b\in R$ with $ab\in I\backslash I^2$, (resp. $ab\in I\backslash I^m, m\geq 3$) either $a\in I$ or $b\in I$. This definition can obviously be made for any commutative ring $R$.  Later, Anderson and Batanieh \cite{AB} gave a generalization of prime ideals which covers all the above mentioned definitions. Let $\phi:\mathfrak{J}(R)\rightarrow \mathfrak{J}(R)\cup\{\emptyset\}$ be a function. A proper ideal $I$ of $R$ is said to be $\phi$-{\it prime} if for $a, b\in R$ with $ab\in I \backslash\phi(I)$, $a\in I$ or $b\in I$.
 Since $I\backslash\phi(I)=I\backslash(I\cap \phi(I))$, without loss of generality we may assume that $\phi(I)\subseteq I$. We henceforth make this assumption.  Weakly primary ideals were first introduced and studied by Ebrahimi Atani and Farzalipour in \cite{EF}. A proper ideal $I$ of $R$ is called {\it weakly primary} if for $a,~b\in R$ with $0\neq ab\in I$, either $a\in I$ or $b\in \sqrt{I}$. In \cite{YF}, Yousefian Darani called a proper ideal $I$ of $R$ to be $\phi$-{\it primary} if for $a, b\in R$ with $ab\in I \backslash\phi(I)$, then either $a\in I$ or $b\in\sqrt{I}$.
He defined the map $\phi_{\alpha}:\mathfrak{J}(R)\rightarrow \mathfrak{J}(R)\cup\{\emptyset\}$ as follows:
\begin{itemize}
\item[(1)] $\phi_{\emptyset}$ : $\phi(I) = \emptyset$ defines primary ideals.
\item[(2)] $\phi_{0}$ : $\phi(I)=0$ defines weakly primary ideals.
\item[(3)] $\phi_{2}$ : $\phi(I)=I^2$ defines almost primary ideals.
\item[(4)] $\phi_{m} (m\geq 2)$ : $\phi(I)=I^m$ defines $m$-almost primary ideals.
\item[(5)] $\phi_{\omega}$ : $\phi(I) = \cap_{m=1}^{\infty} I^m$ defines $\omega$-primary ideals.
\item[(6)] $\phi_{1}$ : $\phi(I) = I$ defines any ideals.
\end{itemize}
Given two functions $\psi_1,\psi_2:\mathfrak{J}(R)\rightarrow \mathfrak{J}(R)\cup\{\emptyset\}$, we define $\psi_1\leq \psi_2$ if $\psi_1(J)\subseteq \psi_2(J)$ for each $J\in\mathfrak{J}(R)$. Note in this case that
$$\phi_{\emptyset}\leq \phi_0\leq \phi_{\omega}\leq\cdots\leq \phi_{m+1}\leq \phi_{m}\leq\cdots\leq \phi_2\leq \phi_1.$$

Badawi in \cite{B} generalized the concept of prime ideals in a different
way. He defined a nonzero proper ideal $I$ of $R$ to be a {\it 2-absorbing
ideal} of $R$ if whenever $a, b, c\in R$ and $abc\in I$, then $ab\in I$ or $%
ac\in I$ or $bc\in I$. Anderson and Badawi 
\cite{AB1} generalized the concept of $2$-absorbing ideals to $n$-absorbing
ideals. According to their definition, a proper ideal $I$ of $R$ is called
an $n$-{\it absorbing} (resp. {\it strongly $n$-absorbing}) ideal if whenever $%
a_1\cdots a_{n+1}\in I$ for $a_1,\dots,a_{n+1}\in R$ (resp. $I_1\cdots
I_{n+1}\subseteq I$ for ideals $I_1,\dots,I_{n+1}$ of $R$), then there are 
$n$ of the $a_i$'s (resp. $n$ of the $I_i$'s) whose product is in $I$. Thus a strongly 1-absorbing ideal is just a prime ideal. Clearly a strongly $n$-absorbing ideal of $R$ is also an $n$-absorbing ideal of $R$. Anderson and Badawi conjectured that these
two concepts are equivalent, e.g., they proved that an ideal $I$ of a Pr\"{u}fer domain $R$ is strongly $n$-absorbing if and only if $I$ is an $n$-absorbing ideal of $R$, \cite[Corollary 6.9]{AB1}. They also gave several results relating strongly $n$-absorbing ideals.
The concept $2$-absorbing ideals has another generalization, called weakly $2$-absorbing
ideals, which has studied in \cite{YB}. A proper ideal $I$ of $R$ is a
{\it weakly 2-absorbing ideal of} $R$ if whenever $a, b, c\in R$ and $0\neq
abc\in I$, then $ab\in I$ or $ac\in I$ or $bc\in I$. Generally, Mostafanasab and Yousefian Darani \cite{MY} called a proper ideal $I$ of $R$ to be a {\it weakly $n$-absorbing} (resp. {\it strongly weakly $n$-absorbing}) ideal if whenever $
0\neq a_1\cdots a_{n+1}\in I$ for $a_1,\dots,a_{n+1}\in R$ (resp. $0\neq I_1\cdots
I_{n+1}\subseteq I$ for ideals $I_1,\dots, I_{n+1}$ of $R$), then there are 
$n$ of the $a_i$'s (resp. $n$ of the $I_i$'s) whose product is in $I$. Clearly a strongly weakly $n$-absorbing ideal of $R$ is also a weakly $n$-absorbing ideal of $R$.  Let $\phi:\mathfrak{J}(R)\rightarrow \mathfrak{J}(R)\cup\{\emptyset\}$ be a function. We say that a proper ideal $I$ of $R$ is a $\phi$-$n$-absorbing ideal of $R$ if $a_1a_2\cdots a_{n+1}\in I\backslash\phi(I)$ for $a_1,a_2,\dots,a_{n+1}\in R$ (resp. $I_1\cdots I_{n+1}\subseteq I\backslash\phi(I)$ for ideals $I_1,\dots, I_{n+1}$ of $R$) implies that there are $n$ of the $a_i$'s (resp. $n$ of the $I_i$'s) whose product is in $I$. Notice that $\phi$-$n$-absorbing ideals of a commutative ring $R$ have already been investigated by Ebrahimpour 
and Nekooei \cite{EN} as $(n,n+1)$-$\phi$-prime ideals.

We call a proper ideal $I$ of $R$ to be a $\phi$-$n$-absorbing primary (resp. strongly $\phi$-$n$-absorbing primary) ideal of $R$ if $a_1a_2\cdots a_{n+1}\in I\backslash\phi(I)$ for some elements $a_1,a_2,\dots,a_{n+1}\in R$ (resp. $I_1\cdots I_{n+1}\subseteq I$ and
$I_1\cdots I_{n+1}\nsubseteq \phi(I)$ for ideals $I_1,\dots, I_{n+1}$ of $R$) implies that either $a_1a_2\cdots a_{n}\in I$ or the product of $a_{n+1}$ with $(n-1)$ of $a_1,a_2,\dots,a_n$ is in $\sqrt{I}$ (resp. either $I_1I_2\cdots I_{n}\subseteq I$ or the product of $I_{n+1}$ with $(n-1)$ of $I_1,I_2,\dots,I_n$ is in $\sqrt{I}$).\\
We can define the map $\phi_{\alpha}:\mathfrak{J}(R)\rightarrow \mathfrak{J}(R)\cup\{\emptyset\}$ as follows:\\
Let $I$ be a $\phi_{\alpha}$-$n$-absorbing primary ideal of $R$. Then
\begin{enumerate}
\item $\phi_{\emptyset}(I)=\emptyset$ $~\Rightarrow~$ $I$ is an $n$-absorbing primary ideal.
\item $\phi_{0}(I)=0$ $~\Rightarrow~$ $I$ is a weakly $n$-absorbing primary ideal.
\item $\phi_{2}(I)=I^2$ $~\Rightarrow~$ $I$ is an almost $n$-absorbing primary ideal.
\item $\phi_{m}(I)=I^m~(m\geq 2)$ $~\Rightarrow~$ $I$ is an $m$-almost $n$-absorbing primary ideal.
\item $\phi_{\omega}(I) = \cap_{m=1}^{\infty} I^m$ $~\Rightarrow~$ $I$ is an $\omega$-$n$-absorbing primary ideal.
\item $\phi_{1}(I) = I$ $~\Rightarrow~$ $I$ is any ideal.
\end{enumerate}

Some of our results use the $R(+)M$ construction. Let $R$ be a ring
and $M$ be an $R$-module. Then $R(+)M = R \times M$ is a ring with identity
$(1, 0)$ under addition defined by $(r, m) + (s, n) = (r + s, m + n)$ and
multiplication defined by $(r, m)(s, n) = (rs, rn + sm)$.

In \cite{Q}, Quartararo et al. said that a commutative ring $R$ is a $u$-{\it ring} provided
$R$ has the property that an ideal contained in a finite union of ideals
must be contained in one of those ideals. They show that every B$\acute{\rm e}$zout ring is a $u$-ring. Moreover, they proved that 
every Pr\"{u}fer domain is a $u$-domain.

Let $R$ be a ring and $\phi:\mathfrak{J}(R)\rightarrow \mathfrak{J}(R)\cup\{\emptyset\}$
be a function. In section 2, we give some basic properties of $\phi$-$n$-absorbing primary ideals. For instance, we prove that
if $\phi$ reverses the inclusion and for every $1\leq i\leq k$, $I_i$ is a $\phi$-$n_i$-absorbing primary ideal of $R$ such that $\sqrt{I_i}$ is a $\phi$-$n_i$-absorbing ideal of $R$, respectively, then $I_1\cap I_2\cap\cdots\cap I_k$ and $I_1I_2\cdots I_k$ are two $\phi$-$n$-absorbing primary ideals of $R$
where $n=n_1+n_2+\cdots+n_k$. It is shown that a Noetherian domain $R$ is a Dedekind domain if
and only if a nonzero $n$-absorbing primary ideal of $R$ is in the form of $I=M_1^{t_1}M_2^{t_2}\cdots M_i^{t_i}$
for some $1\leq i\leq n$ and some distinct maximal ideals $M_1,M_2,\dots, M_i$ of $R$ and some positive integers $t_1,t_2,\dots,t_i$.
Moreover, we prove that if $I$ is an ideal of a ring $R$ such that $\sqrt{I}=M_1\cap M_2\cap\cdots\cap M_n$ where $M_i$'s are maximal ideals of $R$, then $I$ is an $n$-absorbing primary ideal of $R$. We show that if $I$ is an $\phi$-$n$-absorbing primary ideal of $R$ that is not an $n$-absorbing primary
ideal, then $I^{n+1}\subseteq\phi(I)$.

Insection 3, we investigate $\phi$-$n$-absorbing primary ideals of direct products of commutative rings. For example, it is shown that if $R$ is an indecomposable ring and $J$ is a finitely generated $\phi$-$n$-absorbing primary ideal of $R$, where  $\phi\leq\phi_{n+2}$, then $J$ is weakly $n$-absorbing primary. Let $n\geq 2$ be a natural number and $R=R_{1}\times\cdots\times R_{n+1}$ be a decomposable 
ring with identity. Then we prove that $R$ is a von Neumann regular ring if and only if every proper ideal of $R$ is an $n$-almost $n$-absorbing primary ideal of $R$
if and only if every proper ideal of $R$ is an $\omega$-$n$-absorbing primary ideal of $R$.

In Section 4, we study the stability of $\phi$-$n$-absorbing primary ideals with respect
to idealization. As a result of this section, we establish that if $I$ is a proper ideal of $R$ and $M$ is an $R$-module such that $IM=M$, then $I(+)M$ is an $n$-almost $n$-absorbing primary (resp. $\omega$-$n$-absorbing primary) ideal of $R(+)M$ if and only if $I$ is an $n$-almost $n$-absorbing primary  (resp. $\omega$-$n$-absorbing primary) ideal of $R$.

In section 5, we prove that over a $u$-ring $R$ the two concepts of strongly $\phi$-$n$-absorbing primary ideals and of $\phi$-$n$-absorbing primary ideals
are coincide. Moreover, if $R$ is a Pr\"{u}fer domain and $I$ is an ideal of $R$, then $I$ is an $n$-absorbing primary $($resp. a weakly $n$-absorbing primary$)$ ideal of $R$ if and only if $I[X]$ is an $n$-absorbing primary $($resp. a weakly $n$-absorbing primary$)$ ideal of $R[X]$.

\section{Properties of $\phi$-$n$-absorbing primary ideals}
Let $n$ be a positive integer. Consider elements $a_1,\dots,a_n$ and ideals $I_{1},\dots,I_n$ of a ring $R$. Throughout this paper we use the following notations: \\
$a_{1}\cdots \widehat{a_{i}}\cdots a_{n}$: $i$-${th}$ term is excluded from $a_{1}\cdots a_{n}$.\\
Similarly; $I_{1}\cdots \widehat{I_{i}}\cdots I_{n}$: $i$-${th}$ term is excluded from $I_{1}\cdots I_{n}$.\\

 It is obvious that any $n$-absorbing primary ideal of a ring $R$ is a $\phi$-$n$-absorbing primary ideal of $R$. Also it is evident that
 the zero ideal is a weakly $n$-absorbing primary ideal of $R$. Assume that $p_1,p_2,\dots,p_{n+1}$ are distinct prime numbers. We know that the zero ideal $I=\{0\}$ is a weakly $n$-absorbing primary ideal of the ring $\mathbb{Z}_{p_1p_2\cdots p_{n+1}}$. Notice that $p_1p_2\cdots p_{n+1}=0\in I$, but neither $p_1p_2\cdots p_{n}\in I$ nor $p_1\cdots\widehat{p_i}\cdots p_{n+1}\in\sqrt{I}={\rm Nil}(\mathbb{Z}_{p_1p_2\cdots p_{n+1}})$ for every $1\leq i\leq n$. Hence $I$ is not an $n$-absorbing primary ideal of $\mathbb{Z}_{p_1p_2\cdots p_{n+1}}$.
 
\begin{remark}\label{rem1}
Let $I$ be an ideal of a ring $R$ and $\phi:\mathfrak{J}(R)\rightarrow \mathfrak{J}(R)\cup\{\emptyset\}$
be a function.
\begin{enumerate}
\item $I$ is $\phi$-primary if and only if $I$ is $\phi$-1-absorbing primary.
\item If $I$ is $\phi$-$n$-absorbing primary, then it is $\phi$-$i$-absorbing primary for all $i>n$.
\item If $I$ is $\phi$-primary, then it is $\phi$-$n$-absorbing primary for all $n>1$.
\item If $I$ is $\phi$-$n$-absorbing primary for some $n\geq 1$, then there
exists the least $n_0\geq 1$ such that $I$ is $\phi$-$n_0$-absorbing primary. In this case, $I$ is
$\phi$-$n$-absorbing primary for all $n\geq n_0$ and it is not $\phi$-$i$-absorbing primary for $n_0>i>0$.
\end{enumerate}
\end{remark}

\begin{remark}
If $I$ is a radical ideal of a ring $R$, then clearly $I$ is a $\phi$-$n$-absorbing primary (resp. strongly $\phi$-$n$-absorbing primary) ideal if and only if $I$ is a $\phi$-$n$-absorbing (resp. strongly $\phi$-$n$-absorbing) ideal. 
\end{remark}

\begin{theorem}\label{main11}
Let $R$ be a ring and let $\phi:\mathfrak{J}(R)\rightarrow \mathfrak{J}(R)\cup\{\emptyset\}$
be a function. Then the following conditions are equivalent:
\begin{enumerate}
\item $I$ is $\phi$-$n$-absorbing primary;
\item For every elements $x_{1},\dots,x_{n}\in R$ with $x_{1}\cdots x_{n}\notin\sqrt{I}$,
$$(I:_{R}x_{1}\cdots x_{n})\subseteq[\cup_{i=1}^{n-1}(\sqrt{I}:_{R}x_{1}\cdots\widehat{x_{i}}\cdots x_{n})]\cup$$
$$\hspace{4.3cm}(I:_{R}x_{1}\cdots x_{n-1})\cup(\phi(I):_{R}x_{1}\cdots x_{n}).$$ 
\end{enumerate}
\end{theorem}
\begin{proof}
(1)$\Rightarrow$(2) Suppose that $x_{1},\dots,x_{n}\in R$ such that $x_{1}\cdots x_{n}\notin\sqrt{I}$. Let $a\in(I:_{R}x_{1}\cdots x_{n})$. So $ax_{1}\cdots x_{n}\in I$. If $ax_{1}\cdots x_{n}\in\phi(I)$, then $a\in(\phi(I):_Rx_{1}\cdots x_{n})$. Assume that $ax_{1}\cdots x_{n}\notin\phi(I)$. Since $x_{1}\cdots x_{n}\notin\sqrt{I}$, then 
either $ax_1\cdots x_{n-1}\in I$, i.e., $a\in(I:_Rx_1\cdots x_{n-1})$ or for some $1\leq i\leq n-1$ we have $ax_{1}\cdots\widehat{x_{i}}\cdots x_{n}\in\sqrt{I}$, i.e., $a\in(\sqrt{I}:_{R}x_{1}\cdots\widehat{x_{i}}\cdots x_{n})$
or $x_1x_2\cdots x_n\in \sqrt{I}$. By assumption the last case is not hold. Consequently $$(I:_{R}x_{1}\cdots x_{n})\subseteq[\cup_{i=1}^{n-1}(\sqrt{I}:_{R}x_{1}\cdots\widehat{x_{i}}\cdots x_{n})]\cup$$
$$\hspace{4.3cm}(I:_{R}x_{1}\cdots x_{n-1})\cup(\phi(I):_{R}x_{1}\cdots x_{n}).$$\\
(2)$\Rightarrow$(1) Let $a_1a_2\cdots a_{n+1}\in I\backslash\phi(I)$ for some $a_1,a_2,\dots,a_{n+1}\in R$ such that $a_1a_2\cdots a_{n}\\\notin I$.
Then $a_1\in(I:_Ra_2\cdots a_{n+1})$. If $a_2\cdots a_{n+1}\in\sqrt{I}$, then we are done. Hence we may assume that $a_2\cdots a_{n+1}\notin\sqrt{I}$ and
so by part (2), $$(I:_{R}a_{2}\cdots a_{n+1})\subseteq[\cup_{i=2}^{n}(\sqrt{I}:_{R}a_{2}\cdots\widehat{a_{i}}\cdots a_{n+1})]\cup$$
$$\hspace{4.3cm}(I:_{R}a_{2}\cdots a_{n})\cup(\phi(I):_{R}a_{2}\cdots a_{n+1}).$$\\
Since $a_1a_2\cdots a_{n+1}\notin\phi(I)$ and $a_1a_2\cdots a_{n}\notin I$, the only possibility is that $a_1\in\cup_{i=2}^{n}(\sqrt{I}:_{R}a_{2}\cdots\widehat{a_{i}}\cdots a_{n+1})$. Then $a_1a_{2}\cdots\widehat{a_{i}}\cdots a_{n+1}\in\sqrt{I}$ for some $2\leq i\leq n$.
Consequently $I$ is $\phi$-$n$-absorbing primary.
\end{proof}

Let $R$ be an integral domain with quotient field $K$. Badawi and Houston \cite{BH} defined a proper ideal $I$ of $R$ to be strongly primary if, whenever $ab\in I$ with $a, b\in K$, we have $a\in I$ or $b\in \sqrt{I}$. In \cite{YF}, a proper ideal $I$ of $R$ is called strongly $\phi$-primary if whenever $ab\in I\backslash\phi(I)$ with $a,b\in K$, we have either $a\in I$ or $b\in \sqrt{I}$. We say that a proper ideal $I$ of $R$ is {\it quotient $\phi$-$n$-absorbing primary} if whenever $x_1x_2\cdots x_{n+1}\in I$ with $x_1,x_2,\dots,x_{n+1}\in K$, we have either  $x_1x_2\cdots x_{n}\in I$ or $x_1\cdots\widehat{x_i}\cdots x_{n+1}\in\sqrt{I}$ for some $1\leq i\leq n$.

\begin{proposition} 
Let $V$ be a valuation domain with the quotient field $K$, and 
let $\phi:\mathfrak{J}(V )\rightarrow \mathfrak{J}(V )\cup\{\emptyset\}$ be a function. Then every  $\phi$-$n$-absorbing primary ideal of $V$ is quotient $\phi$-$n$-absorbing primary.
\end{proposition}
\begin{proof}
Assume that $I$ is a $\phi$-$n$-absorbing primary ideal of $V$. Let $x_1x_2\cdots x_{n+1}\in I$ for some $x_1,x_2,\dots, x_{n+1}\in K$ such that $x_1x_2\cdots x_{n}\notin I$. If $x_{n+1}\notin V$, then $x_{n+1}^{-1}\in V$, since $V$ is valuation. So $x_1\cdots x_{n}x_{n+1}x_{n+1}^{-1}=x_1\cdots x_{n}\in I$, a contradiction. Hence $x_{n+1}\in V$. If $x_i\in V$ for every $1\leq i\leq n$, then there is nothing to prove. If $x_i\notin V$ for some $1\leq i\leq n$, then $x_1\cdots\widehat{x_i}\cdots x_{n+1}\in I\subseteq\sqrt{I}$. Consequently, $I$ is quotient $\phi$-$n$-absorbing primary.
\end{proof}

\begin{proposition} 
Let $R$ be a von Neumann regular ring and let $\phi:\mathfrak{J}(R)\rightarrow \mathfrak{J}(R)\cup\{\emptyset\}$ 
be a function. Then $I$ is a $\phi$-$n$-absorbing primary ideal of $R$ if and only if $e_1e_2\cdots e_{n+1}\in I\backslash\phi(I)$ for some idempotent elements $e_1,e_2,\dots,e_{n+1}\in R$ implies that either $e_1e_2\cdots e_{n}\in I$ or $e_1\cdots\widehat{e_i}\cdots e_{n+1}\in\sqrt{I}$ for some $1\leq i\leq n$.
\end{proposition}
\begin{proof}
Notice the fact that any finitely generated ideal of a von Neumann regular ring $R$ is generated by an idempotent element.
\end{proof}

\begin{theorem} \label{phi}
Let $R$ be a ring and let $\phi:\mathfrak{J}(R)\rightarrow \mathfrak{J}(R)\cup\{\emptyset\}$ 
be a function. If $I$ is a $\phi$-$n$-absorbing primary ideal of $R$ such that $\sqrt{\phi(I)}=\phi(\sqrt{I})$, then $\sqrt{I}$ is a $\phi$-$n$-absorbing ideal of $R$.
\end{theorem}
\begin{proof}
Let $x_1x_2\cdots x_{n+1}\in\sqrt{I}\backslash\phi(\sqrt{I})$ for some $x_1,x_2,\dots,x_{n+1}\in R$ such that $x_1\cdots\widehat{x_i}\cdots x_{n+1}\notin\sqrt{I}$ for every $1\leq i\leq n$. Then there is a natural number $m$ such that $x_1^mx_2^m\cdots x_{n+1}^m\in I$. If $x_1^mx_2^m\cdots x_{n+1}^m\in\phi(I)$, then $x_1x_2\cdots x_{n+1}\in\sqrt{\phi(I)}=\phi(\sqrt{I})$, which is a contradiction.
Since $I$ is $\phi$-$n$-absorbing primary, our hypothesis implies that $x_1^mx_2^m\cdots x_{n}^m\in I$. Hence
$x_1x_2\cdots x_{n}\in \sqrt{I}$. Therefore $\sqrt{I}$ is a $\phi$-$n$-absorbing ideal of $R$.
\end{proof}

\begin{corollary}\label{minimal}
Let $I$ be an $n$-absorbing primary ideal of $R$. Then $\sqrt{I}=P_1\cap P_2\cap\cdots\cap P_{i}$ where $1\leq i\leq n$ and $P_i$'s are the only distinct prime ideals of $R$ that are minimal over $I$.
\end{corollary}
\begin{proof}
In Theorem \ref{phi}, suppose that $\phi=\phi_{\emptyset}$. Now apply \textrm{\cite[Theorem 2.5]{AB1}}.
\end{proof}

\begin{theorem}\label{cap}
Let $R$ be a ring, and let $\phi:\mathfrak{J}(R)\rightarrow \mathfrak{J}(R)\cup\{\emptyset\}$  be a function
that reverses the inclusion. Suppose that for every $1\leq i\leq k$, $I_i$ is a $\phi$-$n_i$-absorbing primary ideal of $R$ such that $\sqrt{I_i}=P_i$ is a $\phi$-$n_i$-absorbing ideal of $R$, respectively. Set $n:=n_1+n_2+\cdots+n_k$. The following conditions hold:
\begin{enumerate}
\item $I_1\cap I_2\cap\cdots\cap I_k$ is a $\phi$-$n$-absorbing primary ideal of $R$.
\item $I_1I_2\cdots I_k$ is a $\phi$-$n$-absorbing primary ideal of $R$.
\end{enumerate}
\end{theorem}
\begin{proof}
Set $L=I_1\cap I_2\cap\cdots\cap I_k$. Then $\sqrt{L}=P_1\cap P_2\cap\cdots\cap P_k$. Suppose that $a_1a_2\cdots a_{n+1}\in L\backslash\phi(L)$ for some $a_1,a_2,\dots,a_{n+1}\in R$ and $a_1\cdots\widehat{a_i}\cdots a_{n+1}\notin\sqrt{L}$ for every
$1\leq i\leq n$. By, $\sqrt{L}=P_1\cap P_2\cap\cdots\cap P_k$ is $\phi$-$n$-absorbing, then $a_1a_2\cdots a_{n}\in P_1\cap P_2\cap\cdots\cap P_k$. We claim that $a_1a_2\cdots a_{n}\in L$. For every $1\leq i\leq k$, $P_i$ is $\phi$-$n_i$-absorbing
and $a_1a_2\cdots a_{n}\in P_i\backslash\phi(P_i)$, then there exist elements $1\leq{\beta^i_1},{\beta^i_2},\dots,{\beta^i_{n_i}}\leq n$ such that $a_{\beta^i_1}a_{\beta^i_2}\cdots a_{\beta^i_{n_i}}\in P_i$. If $\beta^l_r=\beta^m_s$ for two pairs $l,r$ and $m,s$, then
$$a_{\beta^1_1}a_{\beta^1_2}\cdots a_{\beta^1_{n_1}}\cdots a_{\beta^l_1}a_{\beta^l_2}\cdots a_{\beta^l_r}\cdots a_{\beta^l_{n_l}}\cdots $$
$$\hspace{1cm}a_{\beta^m_1}a_{\beta^m_2}\cdots\widehat{a_{\beta^m_s}}\cdots a_{\beta^m_{n_m}}\cdots a_{\beta^k_1}a_{\beta^k_2}\cdots a_{\beta^k_{n_k}}\in\sqrt{L}.$$
Therefore $a_1\cdots\widehat{a_{\beta^m_s}}\cdots a_na_{n+1}\in\sqrt{L}$, a contradiction. So $\beta^i_j$'s are distinct.
Hence $\{a_{\beta^1_1},a_{\beta^1_2},\dots,a_{\beta^1_{n_1}}, a_{\beta^2_1},a_{\beta^2_2},\dots, a_{\beta^2_{n_2}},\dots,a_{\beta^k_1},a_{\beta^k_2},\dots,a_{\beta^k_{n_k}}\}=\{a_1,a_2,\dots,a_n\}$. If 
$a_{\beta^i_1}a_{\beta^i_2}\cdots a_{\beta^i_{n_i}}\in I_i$ for every $1\leq i\leq k$, then 
$$a_1a_2\cdots a_n=a_{\beta^1_1}a_{\beta^1_2}\cdots a_{\beta^1_{n_1}}a_{\beta^2_1}a_{\beta^2_2}\cdots a_{\beta^2_{n_2}}\cdots a_{\beta^k_1}a_{\beta^k_2}\cdots a_{\beta^k_{n_k}}\in L,$$
thus we are done. Therefore we may assume that $a_{\beta^1_1}a_{\beta^1_2}\cdots a_{\beta^1_{n_1}}\notin I_1$. Since $I_1$ is 
$\phi$-$n_1$-absorbing primary and $$a_{\beta^1_1}a_{\beta^1_2}\cdots a_{\beta^1_{n_1}}a_{\beta^2_1}a_{\beta^2_2}\cdots a_{\beta^2_{n_2}}\cdots a_{\beta^k_1}a_{\beta^k_2}\cdots a_{\beta^k_{n_k}}a_{n+1}=a_1\cdots a_{n+1}\in I_1\backslash\phi(I_1),$$ then we have $a_{\beta^2_1}a_{\beta^2_2}\cdots a_{\beta^2_{n_2}}\cdots a_{\beta^k_1}a_{\beta^k_2}\cdots a_{\beta^k_{n_k}}a_{n+1}\in P_1$. On the other hand $a_{\beta^2_1}a_{\beta^2_2}\cdots a_{\beta^2_{n_2}}\cdots a_{\beta^k_1}a_{\beta^k_2}\cdots a_{\beta^k_{n_k}}a_{n+1}\in P_2\cap\cdots\cap P_k$. Consequently 
$a_{\beta^2_1}a_{\beta^2_2}\cdots a_{\beta^2_{n_2}}\\\cdots a_{\beta^k_1}a_{\beta^k_2}\cdots a_{\beta^k_{n_k}}a_{n+1}\in \sqrt{L}$,
which is a contradiction. Similarly $a_{\beta^i_1}a_{\beta^i_2}\cdots a_{\beta^i_{n_i}}\in I_i$ for every $2\leq i\leq k$. Then
$a_1a_2\cdots a_n\in L$.\\
(2) The proof is similar to that of part (1).
\end{proof}

\begin{corollary}\label{productprime}
Let $R$ be a  ring with $1\neq0$ and let $P_1,P_2,\dots,P_n$ be prime ideals of $R$. Suppose that for every $1\leq i\leq n$,
$P_i^{t_i}$ is a $P_i$-primary ideal of $R$ where $t_i$ is a positive integer. Then $P_1^{t_1}\cap P_2^{t_2}\cap\cdots\cap P_n^{t_n}$
and $P_1^{t_1}P_2^{t_2}\cdots P_n^{t_n}$ are $n$-absorbing primary ideals of $R$. In particular, $P_1\cap P_2\cap\cdots\cap P_n$ and $P_1P_2\cdots P_n$ are $n$-absorbing primary ideals of $R$.
\end{corollary}

\begin{example}
Let $R=\mathbb{Z}[X_2,X_3,\dots,X_n]+3X_1\mathbb{Z}[X_2,X_3,\dots,X_n,X_1]$. 
Set $P_i:=X_{i+1}R$ for $1\leq i\leq n-1$ and $P_n:=3X_1\mathbb{Z}[X_2,X_3,\dots,X_n,X_1]$. Note that for every $1\leq i\leq n$, $P_i$ is a prime ideal of $R$. Let $I=P_1P_2\cdots P_{n-1}P_n^{2}$. Then $3X_1^2.X_2.\cdots .X_n.3=9X_1^2X_2\cdots X_n\in I$ and $3X_1^2.X_2.\cdots .X_n=3X_1^2X_2\cdots X_n\notin I$. On the other hand $X_2.\cdots .X_n.3=3X_2\cdots X_n\notin\sqrt{I}\subseteq P_n$ and $3X_1^2.X_2.\cdots.\widehat{X_i}.\cdots. \\X_n.3=9X_1^2X_2\cdots\widehat{X_i}\cdots X_n\notin\sqrt{I}\subseteq P_{i-1}$ for every $2\leq i\leq n$. Hence $I$ is not $n$-absorbing primary.
\end{example}

In  \textrm{\cite[Example 2.7]{Bt}}, the authors offered an example to show that if $I\subset J$ such that $I$ is a 2-absorbing
primary ideal of $R$ and $\sqrt{I}=\sqrt{J}$, then $J$ need not be a 2-absorbing ideal of $R$. They considered the ideal $J=\langle XYZ,Y^3,X^3\rangle$ of the ring $R=\mathbb{Z}[X,Y,Z]$ and showed that $\sqrt{J}=\langle XY\rangle$. But 
$X\in\sqrt{J}$, which is a contradiction. Therefore their example is incorrect. In the following example we show that
if $I\subset J$ such that $I$ is a $n$-absorbing primary ideal of $R$ and $\sqrt{I}=\sqrt{J}$, then $J$ need not be a $n$-absorbing ideal of $R$.
\begin{example}\label{exam}
Let $R=K[X_1,X_2,\dots,X_{n+2}]$ where $K$ is a field. Consider the ideal
$J=\langle X_1X_2\cdots X_{n+1},X_1^2X_2\cdots X_n,X_1^2X_{n+2}\rangle$ of $R$. Then $$\hspace{-3cm}\sqrt{J}=\langle X_1X_2\cdots X_n, X_1X_{n+2}\rangle$$
$$\hspace{2.3cm}=\langle X_1\rangle\cap\langle X_2,X_{n+2}\rangle\cap\langle X_3,X_{n+2}\rangle\cap\cdots\cap\langle X_n,X_{n+2}\rangle.$$
Set $P_1=\langle X_1\rangle$ and $P_i=\langle X_i,X_{n+2}\rangle$ for every $2\leq i\leq n$. Note that $P_i$'s are prime ideals of $R$. Let
$I=P_1^{2}P_2\cdots P_n$. Then $I\subset J$ and $\sqrt{I}=\sqrt{J}=\cap_{i=1}^{n}P_i$. By Corollary \ref{productprime}, $I$ is an $n$-absorbing primary ideal of $R$, but $J$ is not an $n$-absorbing primary ideal of $R$ because $X_1X_2\cdots X_{n+1}\in J$, but 
$X_1X_2\cdots X_{n}\notin J$ and $X_2\cdots X_{n+1}\notin\sqrt{J}\subseteq\langle X_1\rangle$ and $X_1\cdots\widehat{X_i}\cdots X_{n+1}\notin\sqrt{J}\subseteq\langle X_i,X_{n+2}\rangle$ for every $2\leq i\leq n$.
\end{example}

\begin{theorem}
Let $R$ be a  ring, and let $\phi:\mathfrak{J}(R)\rightarrow \mathfrak{J}(R)\cup\{\emptyset\}$  be a function.
Suppose that $I$ is an ideal of $R$ such that $\sqrt{\phi(\sqrt{I})}\subseteq\phi(I)$. If $\sqrt{I}$ is a $\phi$-$(n-1)$-absorbing
ideal of $R$, then $I$ is a $\phi$-$n$-absorbing primary ideal of $R$.
\end{theorem}
\begin{proof}
Let $\sqrt{I}$ be $\phi$-$(n-1)$-absorbing. Assume that $a_1a_2\cdots a_{n+1}\in I\backslash\phi(I)$ for some $a_1,a_2,\dots,a_{n+1}\in R$ and $a_1a_2\cdots a_{n}\notin I$. Since $$(a_1a_{n+1})(a_2a_{n+1})\cdots(a_na_{n+1})=(a_1a_2\cdots a_{n})a_{n+1}^n\in I\subseteq\sqrt{I}\backslash\phi(\sqrt{I}).$$
Then for some $1\leq i\leq n$, $$(a_1a_{n+1})\cdots\widehat{(a_ia_{n+1})}\cdots(a_na_{n+1})=(a_1\cdots\widehat{a_i}\cdots a_{n})a_{n+1}^{n-1}\in\sqrt{I},$$
and so $a_1\cdots\widehat{a_i}\cdots a_{n}a_{n+1}\in\sqrt{I}.$ Consequently $I$ is $\phi$-$n$-absorbing primary.
\end{proof}

The following example gives an ideal $J$ of a ring $R$ where $\sqrt{J}$ is an $n$-absorbing ideal of $R$, but $J$
is not an $n$-absorbing primary ideal of $R$.
\begin{example}
Let $R=K[X_1,X_2,\dots,X_{n+2}]$ where $K$ is a field and let
$J=\langle X_1X_2\cdots X_{n+1},X_1^2X_2\cdots X_n,X_1^2X_{n+2}\rangle$. Then $$\sqrt{J}=\langle X_1\rangle\cap\langle X_2,X_{n+2}\rangle\cap\langle X_3,X_{n+2}\rangle\cap\cdots\cap\langle X_n,X_{n+2}\rangle.$$
By \textrm{\cite[Theorem 2.1(c)]{AB1}}, $\sqrt{J}$ is an $n$-absorbing ideal of $R$, but $J$ is not an $n$-absorbing primary ideal of $R$ as it is shown in Example \ref{exam}.
\end{example}

We know that if $I$ is an ideal of a ring $R$ such that $\sqrt{I}$ is a maximal ideal of $R$, then $I$ is a primary ideal of $R$.
\begin{theorem} 
Let $I$ be an ideal of a ring $R$. If $\sqrt{I}=M_1\cap M_2\cap\cdots\cap M_n$ where $M_i$'s are maximal ideals of $R$, then 
$I$ is an $n$-absorbing primary ideal of $R$.
\end{theorem}
\begin{proof}
Let $a_1a_2\cdots a_{n+1}\in I$ for some $a_1,a_2,\dots,a_{n+1}\in R$ such that $a_1\cdots\widehat{a_i}\cdots a_{n+1}\\\notin\sqrt{I}$ for every $1\leq i\leq n$. If for some $1\leq i\leq n$, $a_1\cdots\widehat{a_i}\cdots a_{n+1}\in M_j$ (for every $1\leq j\leq n$), then $a_1\cdots\widehat{a_i}\cdots a_{n+1}\in \sqrt{I}$ and so we are done. Without loss of generality we may assume that for every $1\leq i\leq n$,
$a_1\cdots\widehat{a_i}\cdots a_{n+1}\notin M_i$, respectively. Since $M_i$'s are maximal, then 
$M_i+R(a_1\cdots\widehat{a_i}\cdots a_{n+1})=R$ for every $1\leq i\leq n$. Therefore  for every $1\leq i\leq n$
there are $m_i\in M_i$ and $r_i\in R$ such that  $m_i+r_i(a_1\cdots\widehat{a_i}\cdots a_{n+1})=1$. So
$$m_1m_2\cdots m_n+
\sum_{t=1}^{n}\sum^{n-t+1}_{\substack{\alpha_1=1\\\ \alpha_1<\alpha_2<\\\cdots<\alpha_t\leq n}} [r_{\alpha_1}r_{\alpha_2}\cdots r_{\alpha_t}(m_1\cdots\widehat{m_{\alpha_1}}\cdots\widehat{m_{\alpha_2}}\cdots\widehat{m_{\alpha_t}}\cdots m_n)$$
$$\hspace{-2.8cm}\prod_{i=1}^{t}(a_1\cdots\widehat{a_{\alpha_i}}\cdots a_{n+1})]=1$$
Since $m_1m_2\cdots m_n\in \sqrt{I}$, hence $(m_1m_2\cdots m_n)^t\in I$ for some $t\geq 1$. Thus 
$$(m_1m_2\cdots m_n)^t+s[\sum_{t=1}^{n}\sum^{n-t+1}_{\substack{\alpha_1=1\\\ \alpha_1<\alpha_2<\\\cdots<\alpha_t\leq n}} [r_{\alpha_1}r_{\alpha_2}\cdots r_{\alpha_t}(m_1\cdots\widehat{m_{\alpha_1}}\cdots\widehat{m_{\alpha_2}}\cdots\widehat{m_{\alpha_t}}\cdots m_n)$$
$$\hspace{-2.1cm}\prod_{i=1}^{t}(a_1\cdots\widehat{a_{\alpha_i}}\cdots a_{n+1})]]=1$$ for some $s\in R$. Multiply $a_1a_2\cdots a_n$ on both
sides to get 
$$\hspace{-5cm}a_1a_2\cdots a_n=a_1a_2\cdots a_n(m_1m_2\cdots m_n)^t+$$
$$\hspace{2.5cm}s[\sum_{t=1}^{n}\sum^{n-t+1}_{\substack{\alpha_1=1\\\ \alpha_1<\alpha_2<\\\cdots<\alpha_t\leq n}} [r_{\alpha_1}r_{\alpha_2}\cdots r_{\alpha_t}(m_1\cdots\widehat{m_{\alpha_1}}\cdots\widehat{m_{\alpha_2}}\cdots\widehat{m_{\alpha_t}}\cdots m_n)$$
$$\hspace{-.8cm}(a_1a_2\cdots a_n)\prod_{i=1}^{t}(a_1\cdots\widehat{a_{\alpha_i}}\cdots a_{n+1})]]\in I.$$
Hence $I$ is an $n$-absorbing primary ideal.
\end{proof}

Let $R$ be an integral domain with $1\neq0$ and let $K$ be the quotient field of $R$. A nonzero ideal $I$ of $R$ is said to be {\it invertible} if $II^{-1}=R$, where $I^{-1}=\{x\in K\mid xI\subseteq R\}$. An integral domain $R$ is said to be a {\it Dedekind domain} if every
nonzero proper ideal of $R$ is invertible.

\begin{theorem}\label{dedekind}
Let $R$ be a Noetherian integral domain with $1\neq0$ that is not a field . The following conditions are equivalent: 
\begin{enumerate}
\item $R$ is a Dedekind domain;
\item A nonzero proper ideal $I$ of $R$ is an $n$-absorbing primary ideal of $R$ if and only if $I=M_1^{t_1}M_2^{t_2}\cdots M_i^{t_i}$
for some $1\leq i\leq n$ and some distinct maximal ideals $M_1,M_2,\dots,M_i$ of $R$ and some positive integers $t_1,t_2,\dots,t_i$;
\item If $I$ is a nonzero $n$-absorbing primary ideal of $R$, then $I=M_1^{t_1}M_2^{t_2}\cdots M_i^{t_i}$
for some $1\leq i\leq n$ and some distinct maximal ideals $M_1,M_2,\dots, M_i$ of $R$ and some positive integers $t_1,t_2,\dots,t_i$;
\item A nonzero proper ideal $I$ of $R$ is an $n$-absorbing primary ideal of $R$ if and only if $I=P_1^{t_1}P_2^{t_2}\cdots P_i^{t_i}$
for some $1\leq i\leq n$ and some distinct prime ideals $P_1,P_2,\dots,P_i$ of $R$ and some positive integers $t_1,t_2,\dots,t_i$;
\item If $I$ is a nonzero $n$-absorbing primary ideal of $R$, then $I=P_1^{t_1}P_2^{t_2}\cdots P_i^{t_i}$
for some $1\leq i\leq n$ and some distinct prime ideals $P_1,P_2,\dots,P_i$ of $R$ and some positive integers $t_1,t_2,\dots,t_i$.
\end{enumerate}
\end{theorem}
\begin{proof}
(1)$\Rightarrow$(2) Assume that $R$ is a Dedekind domain that is not a field. Then every nonzero prime ideal of $R$ is maximal.
Let $I$ be a nonzero $n$-absorbing primary ideal of $R$. Since $R$ is a Dedekind domain, then there are distinct maximal ideals $M_1,M_2,\dots, M_i$ of $R$ ($k\geq 1$) such that $I=M_1^{t_1}M_2^{t_2}\cdots M_i^{t_i}$ in which $t_j$'s are positive integers. Therefore $\sqrt{I}=M_1\cap M_2\cap\cdots\cap M_i$. Since $I$ is $n$-absorbing primary and every prime ideal of $R$ is maximal, then $\sqrt{I}$ is the intersection of at most $n$ maximal ideals of $R$, by Corollary \ref{minimal}. So $i\leq n$.\\ Conversely, suppose that $I=M_1^{t_1}M_2^{t_2}\cdots M_i^{t_i}$
for some $1\leq i\leq n$ and some distinct maximal ideals $M_1,M_2,\dots,M_i$ of $R$ and some positive integers $t_1,t_2,\dots,t_i$. Then
$I$ is $n$-absorbing primary, by Corollary \ref{productprime}.\\
(1)$\Rightarrow$(4) The proof is similar to that of (1)$\Rightarrow$(2).\\
(2)$\Rightarrow$(3),  (3)$\Rightarrow$(5) and (4)$\Rightarrow$(5) are evident.\\
(5)$\Rightarrow$(1) Let $M$ be an arbitrary maximal ideal of $R$ and $I$ be an ideal of $R$ such that
$M^2\subsetneq I\subsetneq M$. Hence $\sqrt{I}=M$ and so $I$ is $M$-primary. Then $I$ is $n$-absorbing primary, and thus by part (5)
we have that $I=P_1^{t_1}P_2^{t_2}\cdots P_i^{t_i}$ for some $1\leq i\leq n$ and some distinct prime ideals $P_1,P_2,\dots,P_i$ of $R$ and some positive integers $t_1,t_2,\dots,t_i$. Then $\sqrt{I}=P_1\cap P_2\cap\cdots\cap P_i=M$ which shows that $I$ is a power of $M$, a contradiction. Therefore, there are no ideals properly between $M^2$ and $M$. Consequently $R$ is a Dedekind domain, by \textrm{\cite[Theorem 39.2, p. 470]{G}}.
\end{proof}

Since every principal ideal domain is a Dedekind domain, we have the following result as a consequence of Theorem \ref{dedekind}.
\begin{corollary}
Let $R$ be a principal ideal domain and $I$ be a nonzero proper
ideal of $R$. Then $I$ is an $n$-absorbing primary ideal of $R$ if and only if
$I=R(p_1^{t_1}p_2^{t_2}\cdots p_i^{t_i})$, where $p_j$'s are prime elements of $R$, $1\leq i\leq n$ and $t_j$'s are some integers.
\end{corollary}

The following example shows that an $n$-absorbing primary ideal of a ring $R$ need not be of the form
$P_1^{t_1}P_2^{t_2}\cdots P_i^{t_i}$, where $P_j$'s are prime ideals of $R$, $1\leq i\leq n$ and $t_j$'s are some integers.
\begin{example}
Let $R=K[X_1,X_2,\dots,X_n]$ where $K$ is a field and let $I=\langle X_1,X_2,\dots,X_{n-1},X_n^2\rangle$. Since $I$ is
$\langle X_1,X_2,\dots,X_n\rangle$-primary, then $I$ is an $n$-absorbing primary ideal of $R$. But $I$ is not in the form 
of $P_1^{t_1}P_2^{t_2}\cdots P_i^{t_i}$, where $P_j$'s are prime ideals of $R$, $1\leq i\leq n$ and $t_j$'s are some integers.
\end{example}

\begin{theorem}
Let $R$ be a ring, $a\in R$ a nonunit and $m\geq2$ a positive integer.
 If $(0:_Ra)\subseteq\langle a\rangle$, then $\langle a\rangle$ is $\phi$-$n$-absorbing primary, for some $\phi$ 
with $\phi\leq\phi_m$ if and only if $\langle a\rangle$ is $n$-absorbing primary.
\end{theorem}
\begin{proof}
We may assume that $\langle a\rangle$ is $\phi_m$-$n$-absorbing primary. Let $x_1x_2\cdots x_{n+1}\in\langle a\rangle$ for some $x_1,x_2,\dots,x_{n+1}\in R$. If $x_1x_2\cdots x_{n+1}\notin\langle a^m\rangle$, then either $x_1x_2\cdots x_{n}\in\langle a\rangle$ or $x_1\cdots\widehat{x_i}\cdots x_{n+1}\in\sqrt{\langle a\rangle}$ for some $1\leq i\leq n$. Therefore, assume that  
$x_1x_2\cdots x_{n+1}\in\langle a^m\rangle$. Hence $x_1x_2\cdots x_n(x_{n+1}+a)\in\langle a\rangle$. If $x_1x_2\cdots x_n(x_{n+1}+a)\notin\langle a^m\rangle$,
then either $x_1x_2\cdots x_n\in\langle a\rangle$ or $x_1\cdots\widehat{x_i}\cdots x_n(x_{n+1}+a)\in\sqrt{\langle a\rangle}$
for some $1\leq i\leq n$. So, either $x_1x_2\cdots x_n\in\langle a\rangle$ or $x_1\cdots\widehat{x_i}\cdots x_{n+1}\in\sqrt{\langle a\rangle}$ for some $1\leq i\leq n$. Hence, suppose that $x_1x_2\cdots x_n(x_{n+1}+a)\in\langle a^m\rangle$. Thus 
$x_1x_2\cdots x_{n+1}\in\langle a^m\rangle$ implies that $x_1x_2\cdots x_n a\in\langle a^m\rangle$. Therefore,
there exists $r\in R$ such that $x_1x_2\cdots x_n-ra^{m-1}\in(0:_{R}a)\subseteq\langle a\rangle$. Consequently $x_1x_2\cdots x_n\in\langle a\rangle$.
\end{proof}

\begin{corollary}
Let $R$ be an integral domain, $a\in R$ a nonunit element and $m\geq2$ a positive integer. Then $\langle a\rangle$ is $\phi$-$n$-absorbing primary, for some $\phi$ with $\phi\leq\phi_m$ if and only if $\langle a\rangle$ is $n$-absorbing primary.
\end{corollary}

\begin{theorem}
Let $V$ be a valuation domain and $n$ be a natural number. Suppose that $I$ is an ideal of $V$ such that $I^{n+1}$ is not principal.
Then $I$ is a $\phi_{n+1}$-$n$-absorbing primary if and only if it is $n$-absorbing primary.
\end{theorem}
\begin{proof}
$(\Rightarrow)$ Assume that $I$ is $\phi_n$-$n$-absorbing primary that is not $n$-absorbing primary. Therefore 
there are $a_1,\dots,a_{n+1}\in R$ such that $a_1\cdots a_{n+1}\in I$, but neither $a_1\cdots a_n\in I$ nor $a_1\cdots\widehat{a_i}\cdots a_{n+1}\in\sqrt{I}$ for every $1\leq i\leq n$. Hence $\langle a_i\rangle\nsubseteq I$ for every $1\leq i\leq n+1$.
Since $V$ is a valuation domain, thus $I\subset\langle a_i\rangle$ for every $1\leq i\leq n+1$, and so
$I^{n+1}\subseteq\langle a_1\cdots a_{n+1}\rangle$. Since $I^{n+1}$ is not principal, then $a_1\cdots a_{n+1}\in I\backslash I^{n+1}$.
Therefore $I~~~~$ $\phi_{n+1}$-$n$-absorbing primary implies that either $a_1\cdots a_n\in I$ or $a_1\cdots\widehat{a_i}\cdots a_{n+1}\in\sqrt{I}$ for some $1\leq i\leq n$, which is a contradiction. Consequently $I$ is $n$-absorbing primary.\\
$(\Leftarrow)$ is trivial.
\end{proof}

Let $J$ be an ideal of $R$ and $\phi:\mathfrak{J}(R)\rightarrow \mathfrak{J}(R)\cup\{\emptyset\}$ be a function. Define $\phi_J:\mathfrak{I}(R/J)\to \mathfrak{I}(R/J)\cup \{\emptyset\}$ by $\phi_J(I/J) = (\phi(I)+ J)/J$ for every ideal $I\in\mathfrak{J}(R)$ with $J\subseteq I$ (and $\phi_J(I/J) = \emptyset$ if $\phi(I) = \emptyset$). 

\begin{theorem}\label{frac}
Let $J\subseteq I$ be proper ideals of a  ring $R$, and let $\phi:\mathfrak{J}(R)\rightarrow \mathfrak{J}(R)\cup\{\emptyset\}$ be a function. 
\begin{enumerate}
\item If $I$ is a $\phi$-$n$-absorbing primary ideal of $R$, then $I/J$ is a $\phi_J$-$n$-absorbing primary ideal of $R/J$.
\item If $J\subseteq \phi(I)$ and $I/J$ is a $\phi_J$-$n$-absorbing primary ideal of $R/J$, then $I$ is a $\phi$-$n$-absorbing primary
ideal of $R$.
\item If $\phi(I)\subseteq J$ and $I$ is a $\phi$-$n$-absorbing primary ideal of $R$, then $I/J$ is a weakly $n$-absorbing primary ideal of $R/J$.
\item If $\phi(J)\subseteq\phi(I)$, $J$ is a $\phi$-$n$-absorbing primary
ideal of $R$ and  $I/J$ is a weakly $n$-absorbing primary ideal of $R/J$, then $I$ is a $\phi$-$n$-absorbing primary
ideal of $R$.
\end{enumerate}
\end{theorem}

\begin{proof}
(1) Let $a_1, a_2,\dots,a_{n+1}\in R$ be such that $(a_1+J)(a_2+J)\cdots(a_{n+1}+J)\in (I/J)\backslash\phi_J(I/J)=(I/J)\backslash(\phi(I)+ J)/J$. Then $a_1a_2\cdots a_{n+1}\in I\backslash\phi(I)$ and $I$ $\phi$-$n$-absorbing primary gives either $a_1\cdots a_{n}\in I$ or $a_{1}\cdots \widehat{a_{i}}\cdots a_{n+1}\in \sqrt{I}$ for some $1\leq i\leq n$. Therefore either $(a_1+J)\cdots(a_n+J)\in I/J$ or $(a_{1}+J)\cdots\widehat{(a_i+J)}\cdots(a_{n+1}+J)\in \sqrt{I}/J = \sqrt{I/J}$ for some $1\leq i\leq n$. This shows that $I/J$ is $\phi_J$-$n$-absorbing primary.\\
(2) Suppose that $a_{1}a_{2}\cdots a_{n+1}\in I\backslash\phi(I)$ for some $a_{1},a_2,\dots,a_{n+1}\in R$. Then $(a_1+J)(a_2+J)\cdots(a_{n+1}+J)\in(I/J)\backslash(\phi(I)/J) =(I/J)\backslash\phi_J(I/J)$. Since $I/J$ is assumed to be $\phi_J$-$n$-absorbing primary, we get either $(a_1+J)\cdots(a_{n}+J)\in I/J$ or $(a_1+J)\cdots\widehat{(a_i+J)}\cdots(a_{n+1}+J)\in \sqrt{I/J} = \sqrt{I}/J$ for some $1\leq i\leq n$. Consequently, either $a_1\cdots a_{n}\in I$ or $a_{1}\cdots \widehat{a_{i}}\cdots a_{n+1}\in \sqrt{I}$ for some $1\leq i\leq n$, that $I$ is $\phi$-$n$-absorbing primary.\\
(3) is a direct consequence of part (1).\\
(4) Let $a_{1}\cdots a_{n+1}\in I\backslash\phi(I)$ where $a_{1},\dots,a_{n+1}\in R$. Note that $a_{1}\cdots a_{n+1}\notin\phi(J)$ because $\phi(J)\subseteq\phi(I)$. If $a_{1}\cdots a_{n+1}\in J$, then either $a_1\cdots a_n\in J\subseteq I$ or $a_{1}\cdots\widehat{a_{i}}\cdots a_{n+1}\in\sqrt{J}\subseteq\sqrt{I}$ for some $1\leq i\leq n$, since $J$ is $\phi$-$n$-absorbing primary. If $a_{1}\cdots a_{n+1}\notin J$, then $(a_{1}+I)\cdots (a_{n+1}+I)\in(I/J)\backslash\{0\}$ and so either $(a_{1}+I)\cdots (a_{n}+I)\in I/J$ or $(a_{1}+J)\cdots\widehat{(a_{i}+J)}\cdots (a_{n+1}+J)\in \sqrt{I/J}=\sqrt{I}/J$ for some 
$1\leq i\leq n$. Therefore, either $a_1\cdots a_{n}\in I$ or $a_{1}\cdots\widehat{a_{i}}\cdots a_{n+1}\in\sqrt{I}$ for some $1\leq i\leq n$. Consequently $I$ is a  $\phi$-$n$-absorbing primary ideal of $R$.
\end{proof}

\begin{corollary}
Let $R$ be a  ring, and let $\phi:\mathfrak{J}(R)\rightarrow \mathfrak{J}(R)\cup\{\emptyset\}$  be a function. An ideal $I$ of $R$ is $\phi$-$n$-absorbing primary if and only if $I/\phi(I)$ is a weakly $n$-absorbing primary ideal of $R/\phi(I)$.
\end{corollary}
\begin{proof}
In parts (2) and (3) of Theorem \ref{frac} set $J=\phi(I)$.
\end{proof}

\begin{corollary}
Let $R$ be a ring, $\phi:\mathfrak{J}(R)\rightarrow \mathfrak{J}(R)\cup\{\emptyset\}$ a function and $I$ a proper ideal of $R$ such that $\phi(\langle I,X\rangle)\subseteq \langle X\rangle$. Then $\langle I,X\rangle$ is a $\phi$-$n$-absorbing primary ideal of
$R[X]$ if and only if $I$ is a weakly $n$-absorbing ideal of $R$. 
\end{corollary}
\begin{proof}
By parts (3), (4) of Theorem \ref{frac} and regarding the isomorphism $\langle I,X\rangle/\langle X\rangle\\\simeq I$ in $R[X]/\langle X\rangle\simeq R$ we have the result.
\end{proof}

Let $S$ be a multiplicatively closed subset of a  ring $R$. Let $\phi:\mathfrak{J}(R)\rightarrow \mathfrak{J}(R)\cup\{\emptyset\}$ be a function and define $\phi_S: {\mathfrak{I}}(R_S)\to {\mathfrak{I}}(R_S)\cup \{\emptyset\}$ by $\phi_S(J) = (\phi(J\cap R))_S$ (and $\phi_S(J) = \emptyset$ if $\phi(J\cap R) = \emptyset$) for every ideal $J$ of $R_S$. Note that $\phi_S(J)\subseteq J$. \\Let $M$ be an $R$-module. The set of all zero divisors on $M$ is:

${\rm Z}_{R}(M)=\{r\in R\mid \mbox{there exists an element}~0\neq x\in M~\mbox{such that}~rx=0\}$.
\begin{proposition}
Let $R$ be a  ring and $\phi:\mathfrak{J}(R)\rightarrow \mathfrak{J}(R)\cup\{\emptyset\}$ be a function. Suppose that $S$ is a multiplicatively closed subset of $R$ and $I$ is a proper ideal of $R$.
\begin{enumerate}
\item If $I$ is a $\phi$-$n$-absorbing primary ideal of $R$  with $I\cap S=\emptyset$ and $\phi(I)_S\subseteq \phi_S(I_S)$, then $I_S$ is a $\phi_S$-$n$-absorbing primary ideal of $R_S$.
\item If $I_S$ is a $\phi_S$-$n$-absorbing primary ideal of $R_S$ with $\phi_S(I_S)\subseteq\phi(I)_S$, $S\cap Z_R(\frac{I}{\phi(I)})=\emptyset$ and $S\cap Z_R(\frac{R}{I})=\emptyset$, then $I$ is a $\phi$-$n$-absorbing primary ideal of $R$.
\end{enumerate} 
\end{proposition}

\begin{proof}
(1) Assume that $\frac{a_1}{s_1}\frac{a_2}{s_2}\cdots\frac{a_{n+1}}{s_{n+1}}\in I_S\backslash\phi_S(I_S)$ for some $\frac{a_1}{s_1},\frac{a_2}{s_2},\dots,\frac{a_{n+1}}{s_{n+1}}\in R_S$ such that $\frac{a_1}{s_1}\frac{a_2}{s_2}\cdots\frac{a_{n}}{s_{n}}\notin I_S$. Since $\frac{a_1}{s_1}\frac{a_2}{s_2}\cdots\frac{a_{n+1}}{s_{n+1}}\in I_S$, then there is $s\in S$ such that $sa_1a_2\cdots a_{n+1}\in I$. If $sa_1a_2\cdots a_{n+1}\in \phi(I)$, then $\frac{a_1}{s_1}\frac{a_2}{s_2}\cdots\frac{a_{n+1}}{s_{n+1}}=\frac{sa_1a_2\cdots a_{n+1}}{ss_1s_2\cdots s_{n+1}}\in \phi(I)_S\subseteq\phi_S(I_S)$, a contradiction. Hence $a_1a_2\cdots a_n(sa_{n+1})\in I\backslash\phi(I)$. As $I$ is $\phi$-$n$-absorbing primary, we get either $a_1a_2\cdots a_{n}\in I$ or $a_1\cdots\widehat{a_i}\cdots a_n(sa_{n+1})\in\sqrt{I}$ for some $1\leq i\leq n$. The first case implies that $\frac{a_1}{s_1}\frac{a_2}{s_2}\cdots\frac{a_{n}}{s_{n}}\in I_S$ which is a contradiction, and the second case implies that $\frac{a_1}{s_1}\cdots\widehat{\frac{a_i}{s_i}}\cdots\frac{a_{n+1}}{s_{n+1}}\in (\sqrt{I})_S=\sqrt{I_S}$ for some $1\leq i\leq n$. Consequently $I_S$ is a $\phi_S$-$n$-absorbing primary ideal of $R_S$.\\
(2) Let $a_1a_2\cdots a_{n+1}\in I\backslash\phi(I)$ for some $a_1,a_2,\dots,a_{n+1}\in R$ and let $a_1a_2\cdots a_{n}\notin I$. Then $\frac{a_1}{1}\frac{a_2}{1}\cdots \frac{a_{n+1}}{1}\in I_S$. Assume that $\frac{a_1}{1}\frac{a_2}{1}\cdots \frac{a_{n+1}}{1}\in\phi_S(I_S)$. Since $\phi_S(I_S)\subseteq\phi(I)_S$, then there exists a $s\in S$ such that $sa_1a_2\cdots a_{n+1}\in\phi(I)$. Since $S\cap Z_R(\frac{I}{\phi(I)})=\emptyset$ we have that $a_1a_2\cdots a_{n+1}\in\phi(I)$, which is a contradiction. Therefore $\frac{a_1}{1}\frac{a_2}{1}\cdots \frac{a_{n+1}}{1}\in I_S\backslash\phi_S(I_S)$. Hence, either $\frac{a_1}{1}\frac{a_2}{1}\cdots \frac{a_{n}}{1}\in I_S$ or
$\frac{a_1}{1}\cdots\widehat{\frac{a_i}{1}}\cdots \frac{a_{n+1}}{1}\in\sqrt{I_S}=(\sqrt{I})_S$ for some $1\leq i\leq n$.
If $\frac{a_1}{1}\frac{a_2}{1}\cdots \frac{a_{n}}{1}\in I_S$, then there exists $u\in S$ such that $ua_1a_2\cdots a_{n}\in I$ and so
the assumption $S\cap Z_R(\frac{R}{I})=\emptyset$ shows that $a_1a_2\cdots a_{n}\in I$, a contradiction. Therefore, there is 
$1\leq i\leq n$ such that $\frac{a_1}{1}\cdots\widehat{\frac{a_i}{1}}\cdots \frac{a_{n+1}}{1}\in(\sqrt{I})_S$, and thus there is
a $t\in S$ such that $ta_1\cdots\widehat{a_i}\cdots a_{n+1}\in\sqrt{I}$. Note that $S\cap Z_R(\frac{R}{I})=\emptyset$ implies that
$S\cap Z_R(\frac{R}{\sqrt{I}})=\emptyset$, then $a_1\cdots\widehat{a_i}\cdots a_{n+1}\in\sqrt{I}$. Consequently $I$ is a $\phi$-$n$-absorbing primary ideal of $R$.
\end{proof}

Let $f:R\to T$ be a homomorphism of rings and let $\phi_T:\mathfrak{J}(T)\rightarrow \mathfrak{J}(T)\cup\{\emptyset\}$ be a function. Define $\phi_R:\mathfrak{J}(R)\rightarrow \mathfrak{J}(R)\cup\{\emptyset\}$ by $\phi_R(I)=\phi_T(I^e)^c$ (and $\phi_R(I) = \emptyset$ if $\phi_T(I^e)=\emptyset$). We recall that if $R$ is a Pr\"{u}fer domain or $T=R_S$ for some multiplicatively closed subset $S$ of $R$, then for every ideal $J$ of $T$ we have $J^{ce}=J$.
\begin{theorem}\label{context}
Let $f:R\to T$ be a homomorphism of rings. If $J$ is a $\phi_T$-$n$-absorbing primary ideal of $T$ such that $\phi_T(J)\subseteq\phi_T(J^{ce})$ $($e.g. where $J=J^{ce}~)$, then
$J^c$ is a $\phi_R$-$n$-absorbing primary ideal of $R$.
\end{theorem}
\begin{proof}
Let $a_1a_2\cdots a_{n+1}\in J^c\backslash\phi_R(J^c)$ for some $a_1,a_2\dots,a_{n+1}\in R$. If $f(a_1)f(a_2)\\\cdots f(a_{n+1})\in\phi_T(J)$, then $a_1a_2\cdots a_{n+1}\in\phi_T(J)^c\subseteq\phi_T(J^{ce})^c=\phi_R(J^c)$,
which is a contradiction. Therefore $f(a_1)f(a_2)\cdots f(a_{n+1})\in J\backslash\phi_T(J)$.
Hence, either $f(a_1)f(a_2)\cdots f(a_{n})\in J$ or $f(a_1)\cdots\widehat{f(a_i)}\cdots f(a_{n+1})\in \sqrt{J}$ for some
$1\leq i\leq n$. Thus, either $a_1a_2\cdots a_{n}\in J^c$ or $a_1\cdots\widehat{a_i}\cdots a_{n+1}\in \sqrt{J^c}$ for some
$1\leq i\leq n$. Consequently $J^c$ is a $\phi_R$-$n$-absorbing primary ideal of $R$.
\end{proof}

Let $R,~T$ be rings and $\psi_R:\mathfrak{J}(R)\rightarrow \mathfrak{J}(R)\cup\{\emptyset\}$ be a function. Define  $\psi_T:\mathfrak{J}(T)\rightarrow \mathfrak{J}(T)\cup\{\emptyset\}$ by $\psi_T(J)=\psi_R(J^c)^e$ (and $\psi_T(J) = \emptyset$ if $\psi_R(J^c)=\emptyset$). We recall that if $f:R\to T$ is a faithfully flat  homomorphism of rings, then for every ideal $I$ of $R$ we have $I^{ec}=I$.
\begin{theorem}
Let $f:R\to T$ be a faithfully flat  homomorphism of rings.
\begin{enumerate}
\item If $J$ is a $\psi_T$-$n$-absorbing primary ideal of $T$, then $J^c$ is a $\psi_R$-$n$-absorbing primary ideal of $R$.
\item If $I^e$ is a $\psi_T$-$n$-absorbing primary ideal of $T$ for some ideal $I$ or $R$, then $I$ is a $\psi_R$-$n$-absorbing primary ideal of $R$.
\end{enumerate}

\end{theorem}
\begin{proof}
(1) Suppose that $J$ is a $\psi_T$-$n$-absorbing primary ideal of $T$. In Theorem \ref{context} get $\phi_T:=\psi_T$. Let $I$ be an ideal of $R$. Then $$\phi_R(I)=\phi_T(I^e)^c=\psi_T(I^e)^c=\psi_R(I^{ec})^{ec}=\psi_R(I).$$
So $\phi_R=\psi_R$. Moreover, $\psi_T(J)=\psi_R(J^c)^e=\psi_R(J^{cec})^e=\psi_T(J^{ce})$. Therefore $J^c$ is a $\psi_R$-$n$-absorbing primary ideal of $R$.\\
(2) By part (1).
\end{proof}

\begin{proposition}
Let $I$ be an ideal of a ring $R$ such that $\phi(I)$ be an $n$-absorbing primary ideal of $R$. 
If $I$ is a $\phi$-$n$-absorbing primary ideal of $R$, then $I$ is an $n$-absorbing primary ideal of $R$.
\end{proposition}
\begin{proof}
Assume that $a_1a_2\cdots a_{n+1}\in I$ for some elements $a_1,a_2,\dots,a_{n+1}\in R$ such that
$a_1a_2\cdots a_{n}\notin I$. If $a_1a_2\cdots a_{n+1}\in\phi(I)$, then $\phi(I)~$ $n$-absorbing primary and $a_1a_2\cdots a_{n}\notin\phi(I)$ implies that $a_1\cdots\widehat{a_i}\cdots a_{n+1}\in\sqrt{\phi(I)}\subseteq\sqrt{I}$ for some $1\leq i\leq n$, and so we are done. When $a_1a_2\cdots a_{n+1}\notin\phi(I)$ clearly the result follows.
\end{proof}

We say that a $\phi$-prime ideal $P$ of a ring $R$ is a divided $\phi$-prime ideal if $P\subset xR$ for every $x \in R\backslash P$; thus a divided $\phi$-prime ideal is comparable to every ideal of $R$.
\begin{theorem} 
Let $P$ be a divided $\phi$-prime ideal of a ring $R$. Suppose that $I$ is a $\phi$-$n$-absorbing
ideal of $R$ with $\sqrt{I}=P$ and $\phi(P)\subseteq\phi(I)$. Then $I$ is a $\phi$-primary ideal of R.
\end{theorem}
\begin{proof}
Let $xy\in I\backslash\phi(I)$ for $x,y\in R$ and $y\notin P$. Since $xy\in P\backslash\phi(P)$, then $x\in P$. If $y^{n-1}\in\phi(P)$, then $y\in\sqrt{I}=P$, which is a contradiction. Therefore $y^{n-1}\notin\phi(P)$, and so $y^{n-1}\notin P$. Thus $P\subset y^{n-1}R$, because $P$ is a divided $\phi$-prime ideal of $R$. Hence $x=y^{n-1}z$ for some $z\in R$. As $y^{n}z=yx\in I\backslash\phi(I)$, $y^{n}\notin I$, and $I$ is a $\phi$-$n$-absorbing ideal of $R$, we have $x=y^{n-1}z\in I$. Hence $I$ is a $\phi$-primary ideal of $R$. 
\end{proof}

Let $I$ be an ideal of a ring $R$ and $\phi:\mathfrak{J}(R)\rightarrow \mathfrak{J}(R)\cup\{\emptyset\}$  be a function. Assume that $I$ is a $\phi$-$n$-absorbing primary ideal of $R$ and $a_{1},\dots, a_{n+1}\in R$. We say that $(a_{1},\dots,a_{n+1})$
is an $\phi$-$(n+1)$-tuple of $I$ if $a_{1}\cdots a_{n+1}\in\phi(I)$, $a_{1}a_{2}\cdots a_{n}\notin I$ and for each $1\leq i\leq n$, $a_{1}\cdots\widehat{a_{i}}\cdots a_{n+1}\notin\sqrt{I}$.

In the following theorem $a_{1}\cdots\widehat{a_{i}}\cdots\widehat{a_{j}}\cdots a_{n}$ denotes that $a_{i}$ and $a_{j}$ are eliminated from $a_{1}\cdots a_{n}$.

\begin{theorem}\label{first}
Let $I$ be a $\phi$-$n$-absorbing primary ideal of a ring $R$ and suppose that
$(a_{1},\dots,a_{n+1})$ is a $\phi$-$(n+1)$-tuple of $I$ for some $a_{1},\dots, a_{n+1}\in R$. Then for every elements $\alpha_{1},\alpha_{2},\dots,\alpha_{m}\in\{1,2,\dots,n+1\}$ which $1\leq m\leq n$, $$a_{1}\cdots\widehat{a_{\alpha_{1}}}\cdots\widehat{a_{\alpha_{2}}}\cdots\widehat{a_{\alpha_{m}}}\cdots a_{n+1}I^{m}\subseteq\phi(I).$$
\end{theorem}

\begin{proof}
We use induction on $m$. Let $m=1$ and suppose that $a_{1}\cdots\widehat{a_{\alpha_{1}}}\cdots a_{n+1} x\notin\phi(I)$ for some $x\in I$. Then $a_{1}\cdots\widehat{a_{\alpha_{1}}}\cdots a_{n+1}(a_{\alpha_{1}}+x)\notin\phi(I)$. Since $I$ is a $\phi$-$n$-absorbing primary ideal of $R$ and
$a_{1}\cdots\widehat{a_{\alpha_{1}}}\cdots a_{n+1}\notin I$, we conclude that $a_{1}\cdots\widehat{a_{\alpha_{1}}}\cdots\widehat{a_{\alpha_{2}}}\cdots a_{n+1}(a_{\alpha_{1}}+x)\in\sqrt{I}$, for some $1\leq\alpha_{2}\leq n+1$ different from $\alpha_{1}$. Hence $a_{1}\cdots\widehat{a_{\alpha_{2}}}\cdots a_{n+1}\in\sqrt{I}$, a contradiction. Thus $a_{1}\cdots\widehat{a_{\alpha_{1}}}\cdots a_{n+1}I\subseteq\phi(I)$.\\
Now suppose $m>1$ and assume that for all integers less than $m$ the claim holds.
Let $a_{1}\cdots\widehat{a_{\alpha_{1}}}\cdots\widehat{a_{\alpha_{2}}}\cdots\widehat{a_{\alpha_{m}}}\cdots a_{n+1}x_{1}x_{2}\cdots x_{m}\notin\phi(I)$ for some $x_{1},x_{2},\dots,x_{m}\in I$. 
By  induction hypothesis, we conclude that there exists $\zeta\in\phi(I)$ such that

{\small $$
\begin{array}{ll}
a_{1}\cdots\widehat{a_{\alpha_{1}}}\cdots &\widehat{a_{\alpha_{2}}}\cdots\widehat{a_{\alpha_{m}}}\cdots
a_{n+1}(a_{\alpha_{1}}+x_{1})(a_{\alpha_{2}}+x_{2})\cdots(a_{\alpha_{m}}+x_{m})\\
&=\zeta+a_{1}\cdots\widehat{a_{\alpha_{1}}}\cdots\widehat{a_{\alpha_{2}}}\cdots\widehat{a_{\alpha_{m}}}\cdots a_{n+1}x_{1}x_{2}\cdots x_{m}\notin\phi(I).
\end{array}$$}
Now, we consider two cases.\\
{\bf Case 1.} Assume that $\alpha_m<n+1$.
Since $I$ is $\phi$-$n$-absorbing primary, then either
\begin{eqnarray*}
a_{1}\cdots\widehat{a_{\alpha_{1}}}\cdots\widehat{a_{\alpha_{2}}}\cdots\widehat{a_{\alpha_{m}}}\cdots
a_{n}(a_{\alpha_{1}}+x_{1})(a_{\alpha_{2}}+x_{2})\cdots(a_{\alpha_{m}}+x_{m})\in I,
 \end{eqnarray*}
or 
 \begin{eqnarray*}
a_{1}\cdots\widehat{a_{\alpha_{1}}}\cdots\widehat{a_{\alpha_{2}}}\cdots\widehat{a_{\alpha_{m}}}\cdots\widehat{a_{j}}\cdots
a_{n+1}(a_{\alpha_{1}}+x_{1})(a_{\alpha_{2}}+x_{2})\cdots(a_{\alpha_{m}}+x_{m})\\\in\sqrt{I},
 \end{eqnarray*}
for some $j< n+1$ distinct from ${\alpha_{i}}$'s; or
 \begin{eqnarray*}
a_{1}\cdots\widehat{a_{\alpha_{1}}}\cdots\widehat{a_{\alpha_{2}}}\cdots\widehat{a_{\alpha_{m}}}\cdots
a_{n+1}(a_{\alpha_{1}}+x_{1})\cdots\widehat{(a_{\alpha_{i}}+x_{i})}\cdots(a_{\alpha_{m}}+x_{m})\in\sqrt{I}
 \end{eqnarray*}
for some $1\leq i\leq m$.
Thus either $a_1a_2\cdots a_{n}\in I$ or $a_{1}\cdots\widehat{a_{j}}\cdots a_{n+1}\in\sqrt{I}$ or $a_{1}\cdots\widehat{a_{\alpha_{i}}}\cdots a_{n+1}\in\sqrt{I}$,
which any of these cases has a contradiction.\\ 
{\bf Case 2.} Assume that $\alpha_m=n+1$. Since $I$ is 
$\phi$-$n$-absorbing primary, then either 
\begin{eqnarray*}
a_{1}\cdots\widehat{a_{\alpha_{1}}}\cdots\widehat{a_{\alpha_{2}}}\cdots\widehat{a_{\alpha_{m-1}}}\cdots
\widehat{a_{n+1}}(a_{\alpha_{1}}+x_{1})(a_{\alpha_{2}}+x_{2})\cdots\widehat{(a_{\alpha_{m}}+x_{m})}\in I,
\end{eqnarray*}
or 
\begin{eqnarray*}
a_{1}\cdots\widehat{a_{\alpha_{1}}}\cdots\widehat{a_{\alpha_{2}}}\cdots\widehat{a_{\alpha_{m-1}}}\cdots\widehat{a_{j}}\cdots
\widehat{a_{n+1}}(a_{\alpha_{1}}+x_{1})(a_{\alpha_{2}}+x_{2})\cdots(a_{\alpha_{m}}+x_{m})\\\in\sqrt{I},
\end{eqnarray*}
for some $j< n+1$ different from ${\alpha_{i}}$'s; or
\begin{eqnarray*}
a_{1}\cdots\widehat{a_{\alpha_{1}}}\cdots\widehat{a_{\alpha_{2}}}\cdots\widehat{a_{\alpha_{m-1}}}\cdots
\widehat{a_{n+1}}(a_{\alpha_{1}}+x_{1})\cdots\widehat{(a_{\alpha_{i}}+x_{i})}\cdots(a_{\alpha_{m}}+x_{m})\\\in\sqrt{I}
\end{eqnarray*}
for some $1\leq i\leq m-1$.
Thus either $a_1a_2\cdots a_{n}\in I$ or $a_{1}\cdots\widehat{a_{j}}\cdots a_{n+1}\in\sqrt{I}$ or $a_{1}\cdots\widehat{a_{\alpha_{i}}}\cdots a_{n+1}\in\sqrt{I}$,
which any of these cases has a contradiction. Thus $$a_{1}\cdots\widehat{a_{\alpha_{1}}}\cdots\widehat{a_{\alpha_{2}}}\cdots\widehat{a_{\alpha_{m}}}\cdots a_{n+1}I^{m}\subseteq\phi(I).$$

\end{proof}

\begin{theorem}\label{nil}
Let $I$ be an $\phi$-$n$-absorbing primary ideal of $R$ that is not an $n$-absorbing primary
ideal. Then
\begin{enumerate}
\item $I^{n+1}\subseteq\phi(I)$.
\item $\sqrt{I}=\sqrt{\phi(I)}$.
\end{enumerate}
\end{theorem}
\begin{proof}
(1) Since $I$ is not an $n$-absorbing primary ideal of $R$, $I$ has an $\phi$-$(n+1)$-truple-zero $(a_{1},\dots,a_{n+1})$ for some
$a_{1},\dots,a_{n+1}\in R$. Suppose that $x_{1}x_{2}\cdots x_{n+1}\notin\phi(I)$ for some $x_{1},x_{2},\dots,x_{n+1}\in I$. Then by Theorem \ref{first}, there is $\zeta\in\phi(I)$ such that  $(a_{1}+x_{1})\cdots(a_{n+1}+x_{n+1})=\zeta+x_{1}x_{2}\cdots x_{n+1}\notin\phi(I)$. Hence either $(a_{1}+x_{1})\cdots(a_{n}+x_{n})\in I$ or $(a_{1}+x_{1})\cdots\widehat{(a_{i}+x_{i})}\cdots(a_{n+1}+x_{n+1})\in\sqrt{I}$ for some $1\leq i\leq n$. Thus either
$a_{1}\cdots a_{n}\in I$ or $a_{1}\cdots\widehat{a_{i}}\cdots a_{n+1}\in\sqrt{I}$ for some $1\leq i\leq n$, a contradiction. Hence  $I^{n+1}\subseteq\phi(I)$.\\
(2) Clearly, $\sqrt{\phi(I)}\subseteq\sqrt{I}$. As $I^{n+1}\subseteq\phi(I)$, we get $\sqrt{I}\subseteq\sqrt{\phi(I)}$, as required.
\end{proof}

\begin{corollary}\label{nil2}
Let $I$ be an ideal of a ring $R$  that is not $n$-absorbing primary.
\begin{enumerate}
\item If $I$ is weakly $n$-absorbing primary, then $I^{n+1}=\{0\}$ and $\sqrt{I}=Nil(R)$.
\item If $I$ is $\phi$-$n$-absorbing primary where $\phi\leq\phi_{n+2}$, then $I^{n+1}=I^{n+2}$.
\end{enumerate}
\end{corollary}

\begin{corollary}
Let $I$ be a $\phi$-$n$-absorbing primary ideal  where $\phi\leq \phi_{n+2}$. Then $I$ is $\omega$-$n$-absorbing primary.
\end{corollary}
\begin{proof}
If $I$ is $n$-absorbing primary, then it is $\omega$-$n$-absorbing primary. So assume that $I$ is not $n$-absorbing primary. Then $I^{n+1}= I^{n+2}$ by Corollary \ref{nil2}(2).  By hypothesis $I$ is $\phi$-$n$-absorbing primary and $\phi\leq \phi_{n+1}$. So $I$ is $\phi_{n+1}$-$n$-absorbing primary. On the other hand
$\phi_{\omega}(I)=I^{n+1}=\phi_{n+1}(I)$. Therefore $I$ is $\omega$-$n$-absorbing primary.
\end{proof}

\begin{theorem}
Let $R$ be a  ring and let $\phi:\mathfrak{J}(R)\rightarrow \mathfrak{J}(R)\cup\{\emptyset\}$ be a function.
Suppost that $\{I_\lambda\}_{\lambda\in \Lambda}$ is a family of ideals of $R$ such that for every $\lambda,{\lambda^\prime}\in\Lambda$, $\sqrt{\phi(I_\lambda)}=\sqrt{\phi(I_{\lambda^\prime})}$ and $\phi(I_\lambda)\subseteq\phi(I)$. If for every $\lambda\in \Lambda$, $I_\lambda$ is a $\phi$-$n$-absorbing primary ideal of $R$ that is not $n$-absorbing primary, then $I=\bigcap_{\lambda\in \Lambda} I_{\lambda}$ is a  $\phi$-$n$-absorbing primary ideal of $R$.
\end{theorem}
\begin{proof}
Since $I_{\lambda}$'s are $\phi$-$n$-absorbing primary but are not $n$-absorbing primary, then for every $\lambda\in\Lambda$, $\sqrt{I_{\lambda}}=\sqrt{\phi(I_\lambda)}$, by Theorem \ref{nil}. On the other hand $\phi(I_\lambda)\subseteq\phi(I)$ for every $\lambda\in\Lambda$,
and so $\sqrt{\phi(I_\lambda)}\subseteq\sqrt{I}$. Hence $\sqrt{I}=\sqrt{I_{\lambda}}=\sqrt{\phi(I_\lambda)}$ for every $\lambda\in\Lambda$. Let $a_1a_2\cdots a_{n+1}\in I\backslash\phi(I)$ for some $a_1,a_2,\dots,a_{n+1}\in R$, and let $a_1a_2\cdots a_n\notin I$. Therefore there is a $\lambda\in\Lambda$
such that $a_1a_2\cdots a_n\notin I_{\lambda}$. Since $I_{\lambda}$ is $\phi$-$n$-absorbing primary and $a_1a_2\cdots a_{n+1}\in I_\lambda\backslash\phi(I_\lambda)$, then $a_1\cdots\widehat{a_i}\cdots a_{n+1}\in\sqrt{I_{\lambda}}=\sqrt{I}$ for some $1\leq i\leq n$. Consequently $I$ is a $\phi$-$n$-absorbing primary ideal of $R$.
\end{proof}

\begin{corollary}
Let $R$ be a  ring, $\phi:\mathfrak{J}(R)\rightarrow \mathfrak{J}(R)\cup\{\emptyset\}$ be a function and $I$ be an ideal of $R$. Suppose that $\sqrt{\phi(I)}=\phi(\sqrt{I})$ that is an $n$-absorbing ideal of $R$. If $I$ is a $\phi$-$n$-absorbing primary ideal of $R$, then $\sqrt{I}$ is an $n$-absorbing ideal of $R$.
\end{corollary}
\begin{proof}
Assume that $I$ is a $\phi$-$n$-absorbing primary ideal of $R$. If $I$ is an $n$-absorbing primary ideal of $R$, then $\sqrt{I}$ is an $n$-absorbing ideal, by Theorem \ref{phi}. If $I$ is not an $n$-absorbing primary ideal of $R$, then by Theorem \ref{nil} and by our hypothesis, $\sqrt{I}=\sqrt{\phi(I)}$ which is an $n$-absorbing ideal.
\end{proof}

\begin{theorem}
Let $I$ be a $\phi$-$n$-absorbing primary ideal of a ring $R$ that is not $n$-absorbing primary and let $J$ be a $\phi$-$m$-absorbing primary ideal of $R$ that is not $m$-absorbing primary, and $n\geq m$. Suppose that the two ideals $\phi(I)$ and $\phi(J)$ are not coprime. Then
\begin{enumerate}
\item $\sqrt{I+J}=\sqrt{\phi(I)+\phi(J)}$.
\item If $\phi(I)\subseteq J$ and $\phi(J)\subseteq\phi(I+J)$, then $I+J$ is a $\phi$-$n$-absorbing primary ideal of $R$.
\end{enumerate}
\end{theorem}
\begin{proof}
(1) By Theorem \ref{nil}, we have $\sqrt{I}=\sqrt{\phi(I)}$ and $\sqrt{J}=\sqrt{\phi(J)}$. Now, by \cite[2.25(i)]{Sh} the result follows.\\
(2) Assume that $\phi(I)\subseteq J$ and $\phi(J)\subseteq\phi(I+J)$. Since $\phi(I)+\phi(J)\neq R$, then $I+J$
is a proper ideal of $R$, by part (1). Since $(I+J)/J\simeq I/(I\cap J)$ and $I$ is $\phi$-$n$-absorbing primary, we get that $(I+J)/J$ is a weakly $n$-absorbing primary ideal of $R/J$, by Theorem \ref{frac}(3). On the other hand $J$ is also $\phi$-$n$-absorbing primary, by Remark \ref{rem1}(6). Now, the assertion follows from Theorem \ref{frac}(4). 
\end{proof}

Let $R$ be a ring and $M$ an $R$-module. A submodule $N$ of $M$ is called a pure submodule
if the sequence $0\rightarrow N\otimes_{R} E\rightarrow M\otimes_{R} E$ is exact for every $R$-module $E$.

As another consequence of Theorem \ref{nil} we have the following corollary.
\begin{corollary}
Let $R$ be a ring. 
\begin{enumerate}
\item If $I$ is a pure $\phi$-$n$-absorbing primary ideal of R that is not $n$-absorbing primary, then $I=\phi(I)$.
\item If $R$ is von Neumann regular ring, then every $\phi$-$n$-absorbing primary ideal of $R$ that is not
$n$-absorbing primary is of the form $\phi(I)$ for some ideal $I$ of $R$.
\end{enumerate}
\end{corollary}
\begin{proof}
Note that every pure ideal is idempotent (see \cite{F}), also every ideal of a von Neumann regular ring is idempotent.
\end{proof}

\begin{theorem}
Let $R$ be a ring and $\phi:\mathfrak{J}(R)\rightarrow \mathfrak{J}(R)\cup\{\emptyset\}$ be a function. Let
$I$ be a $\phi$-$(n-1)$-absorbing primary ideal of $R$ that is not $(n-1)$-absorbing primary, and $J$ be an ideal of 
$R$ such that $J\subseteq I$ with $\phi(I)\subseteq\phi(J)$. Then $J$ is a $\phi$-$n$-absorbing primary ideal of $R$.
\end{theorem}
\begin{proof}
Since $I$ is a $\phi$-$(n-1)$-absorbing primary ideal that is not $(n-1)$-absorbing primary we have $\sqrt{I}=\sqrt{\phi(I)}$, by Theorem \ref{nil}. Hence
$\sqrt{J}=\sqrt{I}=\sqrt{\phi(I)}$. Let $a_1a_2\cdots a_{n+1}\in J\backslash\phi(J)$ for some $a_1,a_2,\dots, a_{n+1}\in R$ such that $a_1a_2\cdots a_{n}\notin J$.
Since $J\subseteq I$, we have $a_1a_2\cdots a_{n+1}\in I\backslash\phi(I)$. Consider two cases.\\
{\bf Case 1.} Assume that $a_1a_2\cdots a_{n}\notin I$. Since $I$ is $\phi$-$(n-1)$-absorbing primary, then it is $\phi$-$n$-absorbing primary, by Remark \ref{rem1}(6). Hence $a_1\cdots\widehat{a_i}\cdots a_{n+1}\in\sqrt{I}=\sqrt{J}$ for some $1\leq i\leq n$.\\
{\bf Case 2.} Assume that $a_1a_2\cdots a_{n}\in I$. Since $a_1a_2\cdots a_{n+1}\in I\backslash\phi(I)$, we have that $a_1a_2\cdots a_{n}\in I\backslash\phi(I)$.
On the other hand $I$ is a $\phi$-$(n-1)$-absorbing primary ideal, so either $a_1a_2\cdots a_{n-1}\in I\subseteq\sqrt{J}$ or $a_1\cdots\widehat{a_i}\cdots a_{n}\in\sqrt{I}=\sqrt{J}$
for some $1\leq i\leq {n-1}$. Hence $a_1\cdots\widehat{a_i}\cdots a_{n+1}\in\sqrt{J}$ for some $1\leq i\leq n$. Consequently $J$ is a $\phi$-$n$-absorbing primary ideal of $R$.
\end{proof}

\section{$\phi$-$n$-absorbing primary ideals in direct products of commutative rings}

\begin{theorem}
Let $R_1$ and $R_2$ be  rings, and let $I$ be a weakly $n$-absorbing primary ideal of
$R_1$. Then $J=I\times R_2$ is a $\phi$-$n$-absorbing primary ideal of $R=R_1\times R_2$ for each $\phi$ with
$\phi_{\omega}\leq\phi\leq\phi_1$.
\end{theorem}
\begin{proof}
Suppose that $I$ is a weakly $n$-absorbing primary ideal of
$R_1$. If $I$ is $n$-absorbing primary, then $J$ is $n$-absorbing primary and hence is $\phi$-$n$-absorbing primary, for all
$\phi$. Assume that $I$ is not $n$-absorbing primary. Then $I^{n+1}=\{0\}$, Corollary \ref{nil2}(1). Hence $J^{n+1}=\{0\}\times R_2$ and hence $\phi_{\omega}(J)=\{0\}\times R_2$. Therefore, $J\backslash\phi_{\omega}(J)=(I\backslash\{0\})\times R_2$. 
Let $(x_1,y_1)(x_2,y_2)\cdots(x_{n+1},y_{n+1})\in J\backslash\phi_{\omega}(J)$ for some $x_1,x_2,\dots,x_{n+1}\in R_1$
and $y_1,y_2,\dots,y_{n+1}\in R_2$. Then clearly $x_1x_2\cdots x_{n+1}\in I\backslash\{0\}$. Since $I$ is weakly $n$-absorbing primary,
either $x_1\cdots x_n\in I$ or $x_1\cdots\widehat{x_i}\cdots x_{n+1}\in\sqrt{I}$ for some $1\leq i\leq n$. Therefore, either 
$(x_1,y_1)\cdots(x_n,y_n)\\\in J=I\times R_2$ or $(x_1,y_1)\cdots\widehat{(x_i,y_i)}\cdots (x_{n+1},y_{n+1})\in\sqrt{J}=\sqrt{I}\times R_2$ for some $1\leq i\leq n$. Consequently $J$ is a ${\omega}$-$n$-absorbing primary and hence $\phi$-$n$-absorbing primary. 
\end{proof}

\begin{theorem}
Let $R$ be a  ring and $J$ be a finitely generated proper ideal of $R$. Suppose that $J$ is $\phi$-$n$-absorbing primary, where  $\phi\leq\phi_{n+2}$. Then, either $J$ is weakly $n$-absorbing primary or $J^{n+1}\neq 0$ is idempotent and $R$ decomposes as $R_1\times R_2$ where $R_2=J^{n+1}$ and $J=I\times R_2$, where $I$ is weakly $n$-absorbing primary.
\end{theorem}
\begin{proof}
If $J$ is $n$-absorbing primary, then $J$ is weakly $n$-absorbing primary. So we can assume that $J$ is not $n$-absorbing primary. Then by Corollary \ref{nil2}(2),  $J^{n+1}=J^{n+2}$ and hence $J^{n+1}=J^{2(n+1)}$. Thus $J^{n+1}$ is idempotent, since $J^{n+1}$ is finitely generated, $J^{n+1}=\langle e\rangle$ for some idempotent element $e\in R$. Suppose $J^{n+1}=0$. So $\phi(J)=0$, and hence $J$ is weakly $n$-absorbing primary. Assume that $J^{n+1}\neq 0$. Put $R_2=J^{n+1}=Re$ and $R_1=R(1-e)$; hence $R=R_1\times R_2$. Let $I=J(1-e)$, so $J=I\times R_2$, where $I^{n+1}=0$. We show that $I$ is weakly $n$-absorbing primary. Let $x_1,x_2,\dots, x_{n+1}\in R$ and $x_1x_2\cdots x_{n+1}\in I\backslash\{0\}$
such that $x_1x_2\cdots x_{n}\notin I$. So $(x_1,0)(x_2,0)\cdots(x_{n+1},0)=(x_1x_2\cdots x_{n+1},0)\in I\times R_2=J$.
Since $J^{n+1}=\{0\}\times R_2$ and $\phi(J)\subseteq J^{n+1}$, then $(x_1,0)(x_2,0)\cdots(x_{n+1},0)=(x_1x_2\cdots x_{n+1},0)\in J\backslash\phi(J)$. Since $J$ is $\phi$-$n$-absorbing primary, so either $(x_1,0)(x_2,0)\cdots(x_{n},0)=(x_1x_2\cdots x_{n},0)\in I\times R_2=J$ or $(x_1,0)\cdots\widehat{(x_i,0)}\cdots(x_{n+1},0)=(x_1\cdots\widehat{x_i}\cdots x_{n+1},0)\\\in \sqrt{I}\times R_2=\sqrt{J}$ for some $1\leq i\leq n$. The first case implies that $x_1x_2\cdots x_{n}\in I$, which is a contradiction. The second case implies that $x_1\cdots\widehat{x_i}\cdots x_{n+1}\in\sqrt{I}$ for some $1\leq i\leq n$. Consequently $I$ is weakly $n$-absorbing primary.
\end{proof}

\begin{corollary}
Let $R$ be an indecomposable  ring and $J$ a finitely generated $\phi$-$n$-absorbing primary ideal of $R$, where  $\phi\leq\phi_{n+2}$. Then $J$ is weakly $n$-absorbing primary. Furthermore, if $R$ is an integral domain, then $J$ is actually
$n$-absorbing primary.
\end{corollary}

\begin{corollary}
Let $R$ be a Noetherian integral domain. A proper ideal $J$ of $R$ is $n$-absorbing primary if and only if it is $(n+2)$-almost $n$-absorbing primary.
\end{corollary}

\begin{theorem}\label{T4}
Let $R= R_1\times\cdots\times R_s$ be a decomposable  ring and $\psi_i : {\mathfrak{I}}(R_i)\to {\mathfrak{I}}(R_i)\cup \{\emptyset\}$ be a function for $i= 1, 2,\dots,s$. Set $\phi=\psi_1\times\psi_2\times\cdots\times\psi_n$. Suppose that $$L=I_{1}\times\dots\times I_{\alpha_{1}-1}\times R_{\alpha_{1}}\times I_{\alpha_{1}+1}\times\dots\times I_{\alpha_{j}-1}\times R_{\alpha_{j}}\times I_{\alpha_{j}+1}\times\cdots\times I_{s}$$ be an ideal of $R$ in which $\{\alpha_{1},\dots,\alpha_{j}\}\subset\{1,\dots,s\}$. Moreover, suppose that $\psi_{\alpha_i}(R_{\alpha_i})\neq R_{\alpha_i}$ for some $\alpha_i\in\{\alpha_{1},\dots,\alpha_{j}\}$. The following conditions are equivalent:
\begin{enumerate}
\item $L$ is a $\phi$-$n$-absorbing ideal of $R$;
\item $L$ is an $n$-absorbing primary ideal of $R$;
\item $L':=I_{1}\times\dots\times I_{\alpha_{1}-1}\times I_{\alpha_{1}+1}\times\dots\times I_{\alpha_{j}-1}\times I_{\alpha_{j}+1}\times\cdots\times I_{s}$ is an $n$-absorbing primary ideal of 
$$\hspace{1cm}R':=R_{1}\times\dots\times R_{\alpha_{1}-1}\times R_{\alpha_{1}+1}\times\dots\times R_{\alpha_{j}-1}\times R_{\alpha_{j}+1}\times\cdots\times R_{s}.$$
\end{enumerate}
\end{theorem}
\begin{proof}
$(1)\Rightarrow (2)$ Since  $\psi_{\alpha_i}(R_{\alpha_i})\neq R_{\alpha_i}$ for some $\alpha_i\in\{\alpha_{1},\dots,\alpha_{j}\}$, then clearly $L\nsubseteq\sqrt{\phi(L)}$. So by Theorem \ref{nil}(2), $L$ is an $n$-absorbing ideal of $R$.\\ 
$(2)\Rightarrow (3)$ Assume that $L$ is an $n$-absorbing primary ideal of $R$ and

{\small $$
\begin{array}{ll}
(a_{1}^{(1)},\dots,a_{\alpha_{1}-1}^{(1)},&a_{\alpha_{1}+1}^{(1)},\dots,a_{\alpha_{j}-1}^{(1)},
a_{\alpha_{j}+1}^{(1)},\dots,a_{s}^{(1)})\cdots\\
&(a_{1}^{(n+1)},\dots,a_{\alpha_{1}-1}^{(n+1)},
a_{\alpha_{1}+1}^{(n+1)},\dots,a_{\alpha_{j}-1}^{(n+1)},a_{\alpha_{j}+1}^{(n+1)},\dots,a_{s}^{(n+1)})\in L',
\end{array}
$$}
in which $a_{i}^{(t)}$'s are in $R_{i}$, respectively. Then 

{\small $$
\begin{array}{ll}
(a_{1}^{(1)},\dots,a_{\alpha_{1}-1}^{(1)},&1,a_{\alpha_{1}+1}^{(1)},\dots,a_{\alpha_{j}-1}^{(1)},1,
a_{\alpha_{j}+1}^{(1)},\dots,a_{s}^{(1)})\cdots\\
&(a_{1}^{(n+1)},\dots,a_{\alpha_{1}-1}^{(n+1)},1,
a_{\alpha_{1}+1}^{(n+1)},\dots,a_{\alpha_{j}-1}^{(n+1)},1,a_{\alpha_{j}+1}^{(n+1)},\dots,a_{s}^{(n+1)})\in L.
\end{array}
$$}
So, either 
{\small $$
\begin{array}{ll}
(a_{1}^{(1)},\dots,a_{\alpha_{1}-1}^{(1)},&1,a_{\alpha_{1}+1}^{(1)},\dots,a_{\alpha_{j}-1}^{(1)},1,
a_{\alpha_{j}+1}^{(1)},\dots,a_{s}^{(1)})\cdots\\
&(a_{1}^{(n)},\dots,a_{\alpha_{1}-1}^{(n)},1,
a_{\alpha_{1}+1}^{(n)},\dots,a_{\alpha_{j}-1}^{(n)},1,a_{\alpha_{j}+1}^{(n)},\dots,a_{s}^{(n)})\in L.
\end{array}
$$}
or, there exists $1\leq i\leq n$ such that 
$$\hspace{-2cm}(a_{1}^{(1)},\dots,a_{\alpha_{1}-1}^{(1)},1,a_{\alpha_{1}+1}^{(1)},\dots,a_{\alpha_{j}-1}^{(1)},1,
a_{\alpha_{j}+1}^{(1)},\dots,a_{s}^{(1)})\cdots$$
$$\hspace{-1.5cm}(a_{1}^{(i-1)},\dots,a_{\alpha_{1}-1}^{(i-1)},1,a_{\alpha_{1}+1}^{(i-1)},\dots,a_{\alpha_{j}-1}^{(i-1)},1,
a_{\alpha_{j}+1}^{(i-1)},\dots,a_{s}^{(i-1)})$$
$$\hspace{-.6cm}(a_{1}^{(i+1)},\dots,a_{\alpha_{1}-1}^{(i+1)},1,a_{\alpha_{1}+1}^{(i+1)},\dots,a_{\alpha_{j}-1}^{(i+1)},1,
a_{\alpha_{j}+1}^{(i+1)},\dots,a_{s}^{(i+1)})\cdots$$
$$\hspace{1cm}(a_{1}^{(n+1)},\dots,a_{\alpha_{1}-1}^{(n+1)},1,
a_{\alpha_{1}+1}^{(n+1)},\dots,a_{\alpha_{j}-1}^{(n+1)},1,a_{\alpha_{j}+1}^{(n+1)},\dots,a_{s}^{(n+1)})\in\sqrt{L},$$
because $L$ is an $n$-absorbing primary ideal of $R$. Hence, either 
{\small $$
\begin{array}{ll}
(a_{1}^{(1)},\dots,a_{\alpha_{1}-1}^{(1)},&a_{\alpha_{1}+1}^{(1)},\dots,a_{\alpha_{j}-1}^{(1)},
a_{\alpha_{j}+1}^{(1)},\dots,a_{s}^{(1)})\cdots\\
&(a_{1}^{(n)},\dots,a_{\alpha_{1}-1}^{(n)},
a_{\alpha_{1}+1}^{(n)},\dots,a_{\alpha_{j}-1}^{(n)},a_{\alpha_{j}+1}^{(n)},\dots,a_{s}^{(n)})\in L'.
\end{array}
$$}
or there exists $1\leq i\leq n$ such that 
$$\hspace{-2cm}(a_{1}^{(1)},\dots,a_{\alpha_{1}-1}^{(1)},a_{\alpha_{1}+1}^{(1)},\dots,a_{\alpha_{j}-1}^{(1)},
a_{\alpha_{j}+1}^{(1)},\dots,a_{s}^{(1)})\cdots$$
$$\hspace{-1.5cm}(a_{1}^{(i-1)},\dots,a_{\alpha_{1}-1}^{(i-1)},a_{\alpha_{1}+1}^{(i-1)},\dots,a_{\alpha_{j}-1}^{(i-1)},
a_{\alpha_{j}+1}^{(i-1)},\dots,a_{s}^{(i-1)})$$
$$\hspace{-.6cm}(a_{1}^{(i+1)},\dots,a_{\alpha_{1}-1}^{(i+1)},a_{\alpha_{1}+1}^{(i+1)},\dots,a_{\alpha_{j}-1}^{(i+1)},
a_{\alpha_{j}+1}^{(i+1)},\dots,a_{s}^{(i+1)})\cdots$$
$$\hspace{1cm}(a_{1}^{(n+1)},\dots,a_{\alpha_{1}-1}^{(n+1)},
a_{\alpha_{1}+1}^{(n+1)},\dots,a_{\alpha_{j}-1}^{(n+1)},a_{\alpha_{j}+1}^{(n+1)},\dots,a_{s}^{(n+1)})\in\sqrt{L'},$$
Consequently, $L'$ is an $n$-absorbing primary ideal of $R'$.\\
$(3)\Rightarrow (1)$ Let $L'$ is an $n$-absorbing primary ideal of $R'$. It is routine to see that $L$ is an $n$-absorbing primary ideal of $R$. Consequently, $L$ is a $\phi$-$n$-absorbing primary ideal of $R$.
\end{proof}

\begin{theorem}\label{prod}
Let $R=R_{1}\times\cdots\times R_{n}$ be a  ring with
identity and let $\psi_i : {\mathfrak{I}}(R_i)\to {\mathfrak{I}}(R_i)\cup \{\emptyset\}$ be a function for $i= 1, 2,\dots,n$ such that $\psi_{n}(R_{n})\neq R_n$. Set $\phi=\psi_1\times\psi_2\times\cdots\times\psi_n$. Suppose that $I_{1}\times I_{2}\times\cdots\times I_{n}$ is an ideal of $R$ which $\psi_{1}(I_{1})\neq I_1$, and for some $2\leq j\leq n$, $\psi_{j}(I_{j})\neq I_j$, and $I_{i}$ is a proper ideal of $R_{i}$ for each $1\leq i\leq n-1$. The following conditions are equivalent:
\begin{enumerate}
\item $I_{1}\times I_{2}\times\cdots\times I_{n}$ is a  $\phi$-$n$-absorbing primary ideal of $R$;
\item $I_{n}=R_{n}$ and $I_{1}\times I_{2}\times\cdots\times I_{n-1}$ is an $n$-absorbing primary ideal of $R_{1}\times\cdots\times R_{n-1}$ or $I_{n}$ is a primary ideal of $R_{n}$ and for every $1\leq i\leq n-1$, $I_{i}$ is a primary ideal of $R_{i}$, respectively;
\item $I_{1}\times I_{2}\times\cdots\times I_{n}$ is an $n$-absorbing primary ideal of $R$.
\end{enumerate}
\end{theorem}
\begin{proof}
(1)$\Rightarrow$(2) Suppose that $I_{1}\times I_{2}\times\cdots\times I_{n}$ is a  $\phi$-$n$-absorbing primary ideal of $R$. First assume that $I_{n}=R_{n}$. Since $\psi_{n}(R_{n})\neq R_n$, then $I_{1}\times I_{2}\times\cdots\times I_{n-1}$ is an $n$-absorbing primary ideal of $R_{1}\times\cdots\times R_{n-1}$, by Theorem \ref{T4}. Now, suppose that $I_{n}\neq R_{n}$. Fix $2\leq i\leq n$. We show that $I_{i}$ is a primary ideal of $R_i$. Suppose that $ab\in I_{i}$ for some $a,b\in R_{i}$. Let $x\in I_1\backslash\psi_1(I_1)$. Then
$$\hspace{-3cm}(x,1,\dots,1)(1,0,1,\dots,1,\dots,1)(1,1,0,1,\dots,1,\dots,1)\cdots$$
$$\hspace{-1.5cm}(1,\dots,1,0,\overbrace{1}^{i-th},\dots,1)(1,\dots,\overbrace{1}^{i-th},0,1,\dots,1)\cdots(1,\dots,1,0)$$
$$\hspace{-3cm}(1,\dots,1,\overbrace{a}^{i-th},1,\dots,1)
(1,\dots,1,\overbrace{b}^{i-th},1,\dots,1)$$
$$\hspace{1cm}=(x,0,\dots,0,\overbrace{ab}^{i-th},0,\dots,0)
\in I_{1}\times\dots\times I_{n}\backslash\psi_1(I_1)\times\cdots\times\psi_n(I_n),$$
Since $I_{1}\times I_{2}\times\cdots\times I_{n}$ is  $\phi$-$n$-absorbing primary and $I_{i}$'s are proper, then either
$$\hspace{-2.7cm}(x,1,\dots,1)(1,0,1,\dots,1,\dots,1)(1,1,0,1,\dots,1,\dots,1)\cdots$$
$$\hspace{-1.5cm}(1,\dots,1,0,\overbrace{1}^{i-th},\dots,1)(1,\dots,\overbrace{1}^{i-th},0,1,\dots,1)\cdots(1,\dots,1,0)$$
$$(1,\dots,1,\overbrace{a}^{i-th},1,\dots,1)=(x,0,\dots,0,\overbrace{a}^{i-th},0,\dots,0)\in I_{1}\times\dots\times I_{n},$$
or
$$\hspace{-2.6cm}(x,1,\dots,1)(1,0,1,\dots,1,\dots,1)(1,1,0,1,\dots,1,\dots,1)\cdots$$
$$\hspace{-1.5cm}(1,\dots,1,0,\overbrace{1}^{i-th},\dots,1)(1,\dots,\overbrace{1}^{i-th},0,1,\dots,1)\cdots(1,\dots,1,0)$$
$$\hspace{.7cm}(1,\dots,1,\overbrace{b}^{i-th},1,\dots,1)=(x,0,\dots,0,\overbrace{b}^{i-th},0,\dots,0)\in\sqrt{I_{1}\times\dots\times I_{n}},$$
and thus either $a\in I_i$ or $b\in\sqrt{I_i}$. Consequently $I_i$ is a primary ideal of $R_i$. Since for some $2\leq j\leq n$, $\psi_j(I_j)\neq I_j$, similarly we can show that $I_{1}$ is a primary ideal of $R_{1}$.\\
(2)$\Rightarrow$(3) If $I_{n}=R_{n}$ and $I_{1}\times I_{2}\times\cdots\times I_{n-1}$ is an $n$-absorbing primary ideal of $R_{1}\times\cdots\times R_{n-1}$, then $I_{1}\times I_{2}\times\cdots\times I_{n}$ is an $n$-absorbing primary ideal of $R$, by Theorem \ref{T4}. Now, 
assume that $I_{n}$ is a primary ideal of $R_{n}$ and for each $1\leq i\leq n-1$, $I_{i}$ is a primary ideal of $R_{i}$. Suppose that 
\[(a_{1}^{(1)},\dots,a_{n}^{(1)})(a_{1}^{(2)},\dots,a_{n}^{(2)})\cdots(a_{1}^{(n+1)},\dots,a_{n}^{(n+1)})\]
\[\in I_{1}\times I_{2}\times\cdots\times I_{n}\backslash \psi_1(I_1)\times\cdots\times\psi_n(I_n),\]  
in which for every $1\leq j\leq n+1$, $a_{i}^{(j)}$'s are in $R_{i}$, respectively. Suppose that 
\[(a_{1}^{(1)},\dots,a_{n}^{(1)})(a_{1}^{(2)},\dots,a_{n}^{(2)})\cdots(a_{1}^{(n)},\dots,a_{n}^{(n)})\notin I_{1}\times I_{2}\times\cdots\times I_{n}.\]  
Without loss of generality we may assume that $a_{1}^{(1)}\cdots a_{n}^{(n)}\notin I_1$. Since $I_1$ is primary, we deduce that 
$a_{1}^{(n+1)}\in\sqrt{I_1}$.
On the other hand $\sqrt{I_i}$ is a prime ideal, for any $2\leq i\leq n$, then at least one of the $a_{i}^{(j)}$'s is in $\sqrt{I_{i}}$, say $a_{i}^{(i)}\in\sqrt{I_i}$. Thus $(a_{1}^{(2)},\dots,a_{n}^{(2)})\cdots(a_{1}^{(n+1)},\dots,a_{n}^{(n+1)})\in\sqrt{I_{1}\times I_{2}\times\cdots\times I_{n}}$.
Consequently $I_{1}\times I_{2}\times\cdots\times I_{n}$ is an $n$-absorbing primary ideal of $R$. \\
(3)$\Rightarrow$(1) is obvious.
\end{proof}

\begin{theorem}\label{prod2}
Let $R=R_{1}\times\cdots\times R_{n}$ be a  ring with
identity and let $\psi_i : {\mathfrak{I}}(R_i)\to {\mathfrak{I}}(R_i)\cup \{\emptyset\}$ be a function for $i= 1, 2,\dots,n$ such that $\psi_{n}(R_{n})\neq R_n$. Set $\phi=\psi_1\times\psi_2\times\cdots\times\psi_n$, and suppose that for every $1\leq i\leq n-1$, $I_{i}$ is a proper ideal of $R_{i}$ such that $\psi_1(I_1)\neq I_{1}$ and $I_{n}$ is an ideal of $R_{n}$. Consider the following conditions:
\begin{enumerate}
\item $I_{1}\times\dots\times I_{n}$ is a  $\phi$-$n$-absorbing primary ideal of $R$ that is not an $n$-absorbing primary ideal of $R$.
\item $I_{1}$ is a  $\psi_1$-primary ideal of $R_{1}$ that is not a primary ideal and for every $2\leq i\leq n$, $I_{i}=\psi_i(I_i)$ is a primary ideal of $R_{i}$, respectively.
\end{enumerate}
Then $(1)\Rightarrow(2)$.
\end{theorem}
\begin{proof}
(1)$\Rightarrow$(2) Assume that $I_{1}\times\dots\times I_{n}$ is a  $\phi$-$n$-absorbing primary ideal of $R$ that is not an $n$-absorbing primary ideal. If for some $2\leq i\leq n$ we have $\psi_i(I_i)\neq I_i$, then $I_{1}\times\dots\times I_{n}$ is an $n$-absorbing primary ideal of $R$ by Theorem \ref{prod}, which contradicts our assumption. Thus for every $2\leq i\leq n$, $\psi_i(I_i)=I_i$ and so $I_n\neq R_n$. A proof similar to part (1)$\Rightarrow$(2) of Theorem \ref{prod} shows that for every $2\leq i\leq n$, $\psi_i(I_i)=I_i$ is a primary ideal of $R_i$. Now, we show that $I_1$ is a  $\psi_1$-primary ideal of $R_1$. Consider $a,b\in R_1$ such that $ab\in I_1\backslash\psi_1(I_1)$. Note that 
\[(1,0,1,\dots,1)(1,1,0,1,\dots,1)\cdots(1,\dots,1,0)(a,1,\dots,1)(b,1,\dots,1)\]
\[\hspace{1.3cm}{=(ab,0,\dots,0)\in (I_{1}\times I_2\times\cdots\times I_n)\backslash(\psi_1(I_1)\times\cdots\times\psi_n(I_n)).}\]
Because $I_i$'s are proper, the product of \ $(a,1,\dots,1)(b,1,\dots,1)$ with $n-2$ of $(1,0,1,\dots,1), (1,1,0,1,\dots,1),\dots,(1,\dots,1,0)$ is not in $\sqrt{I_{1}\times I_2\times\cdots\times I_n}$.
Since $I_{1}\times I_2\times\cdots\times I_n$ is a $\phi$-$n$-absorbing primary ideal of $R$, we have either 
\[(1,0,1,\dots,1)(1,1,0,1,\dots,1)\cdots(1,\dots,1,0)(a,1,\dots,1)\]
\[=(a,0,\dots,0)\in I_{1}\times I_2\times\cdots\times I_n.\] 
or 
\[(1,0,1,\dots,1)(1,1,0,1,\dots,1)\cdots(1,\dots,1,0)(b,1,\dots,1)\]
\[=(b,0,\dots,0)\in\sqrt{I_{1}\times I_2\times\cdots\times I_n}.\] 
So either $a\in I_1$ or $b\in\sqrt{I_1}$. Thus $I_1$ is a $\psi_1$-primary ideal of $R_1$. Assume $I_1$ is a primary ideal of $R_1$, since for every $2\leq i\leq n$, $I_i$ is a 
primary ideal of $R_i$, it is easy to see that $I_{1}\times\dots\times I_{n}$ is an
 $n$-absorbing primary ideal of $R$, which is a contradiction.
 \end{proof}

\begin{theorem}
Let $R=R_{1}\times R_2\times R_{3}$ be a  ring with
identity and let $\psi_i : {\mathfrak{I}}(R_i)\to {\mathfrak{I}}(R_i)\cup \{\emptyset\}$ be a function for $i= 1, 2,3$ such that $\psi_{3}(R_{3})\neq R_3$. Set $\phi=\psi_1\times\psi_2\times\psi_3$, and suppose that for $i=1,2$, $I_{i}$ is a proper ideal of $R_{i}$ such that $\psi_1(I_1)\neq I_{1}$ and $I_{3}$ is an ideal of $R_{3}$. The following conditions are equivalent:
\begin{enumerate}
\item $I_{1}\times I_2\times I_{3}$ is a  $\phi$-$3$-absorbing primary ideal of $R$ that is not an $3$-absorbing primary ideal of $R$.
\item $I_{1}$ is a  $\psi_1$-primary ideal of $R_{1}$ that is not a primary ideal and for $ i=2,3$, $I_{i}=\psi_i(I_i)$ is a primary ideal of $R_{i}$, respectively.
\end{enumerate}
\end{theorem}
\begin{proof}
(1)$\Rightarrow$(2) See Theorem \ref{prod2}.\\
(2)$\Rightarrow$(1) Let 
\[(a_1^{(1)},a_2^{(1)},a_3^{(1)})(a_1^{(2)},a_2^{(2)},a_3^{(2)})(a_1^{(3)},a_2^{(3)},a_3^{(3)})(a_1^{(4)},a_2^{(4)},a_3^{(4)})\]
\[\in I_1\times\psi_2(I_2)\times\psi_3(I_3)\backslash\psi_1(I_1)\times\psi_2(I_2)\times\psi_3(I_3),\]
in which, for every $1\leq i\leq3$, $a_i^{(1)},a_i^{(2)},a_i^{(3)},a_i^{(4)}\in R_i$, respectively. Assume that 
\[(a_1^{(1)},a_2^{(1)},a_3^{(1)})(a_1^{(2)},a_2^{(2)},a_3^{(2)})(a_1^{(3)},a_2^{(3)},a_3^{(3)})\notin I_1\times\psi_2(I_2)\times\psi_3(I_3).\]
Therefore, we may have three cases; $a_1^{(1)}a_1^{(2)}a_1^{(3)}\notin I_1$ or $a_2^{(1)}a_2^{(2)}a_2^{(3)}\notin I_2$ or
$a_3^{(1)}a_3^{(2)}a_3^{(3)}\notin I_3$.\\
{\bf Case 1.} Assume that $a_1^{(1)}a_1^{(2)}a_1^{(3)}\notin I_1$. Since $a_1^{(1)}a_1^{(2)}a_1^{(3)}a_1^{(4)}\in I_1\backslash\psi_1(I_1)$ and $I_1$ is $\psi_1$-primary, then $a_1^{(4)}\in\sqrt{I_1}$. On the other hand $\sqrt{I_2}$ and $\sqrt{I_3}$
are prime and $a_2^{(1)}a_2^{(2)}a_2^{(3)}a_2^{(4)}\in\sqrt{I_2}$ and 
$a_3^{(1)}a_3^{(2)}a_3^{(3)}a_3^{(4)}\in\sqrt{I_3}$. Without loss of generality we may assume that $a_2^{(2)}\in\sqrt{I_2}$ and $a_3^{(3)}\in\sqrt{I_3}$. Hence 
\[(a_1^{(2)},a_2^{(2)},a_3^{(2)})(a_1^{(3)},a_2^{(3)},a_3^{(3)})(a_1^{(4)},a_2^{(4)},a_3^{(4)})\in \sqrt{I_1\times\psi_2(I_2)\times\psi_3(I_3)}.\]
{\bf Case 2.} Assume that $a_2^{(1)}a_2^{(2)}a_2^{(3)}\notin I_2$. Since $I_2$ is primary and $a_2^{(1)}a_2^{(2)}a_2^{(3)}a_2^{(4)}\in I_2$, then $a_2^{(4)}\in\sqrt{I_2}$. Since $\sqrt{I_3}$ is prime and
$a_3^{(1)}a_3^{(2)}a_3^{(3)}a_3^{(4)}\in\sqrt{I_3}$, then at least one of the $a_3^{(j)}$'s is in $\sqrt{I_3}$. Let  
$a_3^{(1)}\in\sqrt{I_3}$. Since $I_1$ is $\psi_1$-primary and $a_1^{(1)}a_1^{(2)}a_1^{(3)}\in I_1\backslash\psi_1(I_1)$,
then either $a_1^{(3)}\in I_1$ or $a_1^{(1)}a_1^{(2)}\in\sqrt{I_1}$. So
\[(a_1^{(1)},a_2^{(1)},a_3^{(1)})(a_1^{(3)},a_2^{(3)},a_3^{(3)})(a_1^{(4)},a_2^{(4)},a_3^{(4)})\in \sqrt{I_1\times\psi_2(I_2)\times\psi_3(I_3)}\]
or 
\[(a_1^{(1)},a_2^{(1)},a_3^{(1)})(a_1^{(2)},a_2^{(2)},a_3^{(2)})(a_1^{(4)},a_2^{(4)},a_3^{(4)})\in \sqrt{I_1\times\psi_2(I_2)\times\psi_3(I_3)}.\]
Let $a_3^{(2)}\in\sqrt{I_3}$. Since $I_1$ is $\psi_1$-primary, we have either $a_1^{(1)}\in I_1$ or $a_1^{(2)}a_1^{(3)}\in\sqrt{I_1}$.
Hence 
\[(a_1^{(1)},a_2^{(1)},a_3^{(1)})(a_1^{(2)},a_2^{(2)},a_3^{(2)})(a_1^{(4)},a_2^{(4)},a_3^{(4)})\in \sqrt{I_1\times\psi_2(I_2)\times\psi_3(I_3)}\]
or 
\[(a_1^{(2)},a_2^{(2)},a_3^{(2)})(a_1^{(3)},a_2^{(3)},a_3^{(3)})(a_1^{(4)},a_2^{(4)},a_3^{(4)})\in \sqrt{I_1\times\psi_2(I_2)\times\psi_3(I_3)}.\]
With a similar manner we can investigate the other remaining situations.\\
{\bf Case 3.} For $a_3^{(1)}a_3^{(2)}a_3^{(3)}\notin I_3$ we act like case 2.
Consequently $I_{1}\times I_2\times I_{3}$ is a  $\phi$-$3$-absorbing primary ideal of $R$. Since $I_1$ is not a primary ideal of $R_{1}$, there exist elements $a,b\in R_{1}$ such that $ab\in\psi_1(I_1)$,
but $a\notin I_1$ and $b\notin\sqrt{I_1}$. Hence $$(1,0,1)(1,1,0)(a,1,1)(b,1,1)=(ab,0,0)\in\psi_1(I_1)\times\psi_2(I_2)\times\psi_3(I_3),$$ but neither $$(1,0,1)(1,1,0)(a,1,1)=(a,0,0)\in I_1\times I_2\times I_3,$$ nor $$(1,0,1)(1,1,0)(b,1,1)=(b,0,0)\in\sqrt{I_1\times I_2\times I_3},$$ also the product of $(a,1,1)(b,1,1)$ with any one of elements $(1,0,1),(1,1,0)$ is not in $\sqrt{I_{1}\times I_2 \times I_3}$. Consequently $I_{1}\times I_2 \times I_3$ is not a $3$-absorbing primary ideal of $R$.

\end{proof}

\begin{theorem}\label{nonzero}
Let $R=R_{1}\times\cdots\times R_{n+1}$ where $R_{i}$'s are 
rings with identity and let  for $i= 1, 2,\dots,n+1$, $\psi_i : {\mathfrak{I}}(R_i)\to {\mathfrak{I}}(R_i)\cup \{\emptyset\}$ be a function such that $\psi_{i}(R_{i})\neq R_i$. Set $\phi=\psi_1\times\psi_2\times\cdots\times\psi_{n+1}$.
\begin{enumerate}
\item For every ideal $I$ of $R$, $\phi(I)$ is not an $n$-absorbing primary ideal of $R$;
\item If $I$ is a $\phi$-$n$-absorbing primary ideal of $R$, then either $I=\phi(I)$, or $I$ is an $n$-absorbing primary ideal of $R$.
\end{enumerate}
\end{theorem}
\begin{proof}
Let $I$ be an ideal of $R$. We know that the ideal $I$ is of the form $I_{1}\times\cdots\times I_{n+1}$ where $I_{i}$'s are ideals of
$R_{i}$'s, respectively. \\
(1) Suppose that $\phi(I)$ is an $n$-absorbing primary ideal of $R$. Since
$$(0,1,\dots,1)(1,0,1,\dots,1)\cdots(1,\dots,1,0)=(0,\dots,0)$$
$$\in\phi(I)=\psi_1(I_1)\times\cdots\times\psi_{n+1}(I_{n+1}),$$
we have that either 
$$(0,1,\dots,1)(1,0,1,\dots,1)\cdots(1,\dots,1,0,1)=(0,\dots,0,1)$$
$$\in\psi_1(I_1)\times\cdots\times\psi_{n+1}(I_{n+1}),$$
or the product of $(1,\dots,1,0)$ with $n-1$ of $(0,1,\dots,1),(1,0,1,\dots,1),\dots,\\(1,\dots,1,0,1)$ is in $\sqrt{\phi(I)}$. Hence, for some 
$1\leq i\leq n+1$, $1\in\psi_i(I_i)$ which implies that $\psi_i(R_i)=R_i$, a contradiction. Consequently $\phi(I)$ is not an $n$-absorbing primary ideal of $R$.\\
(2) If $I=\phi(I)$, then clearly $I$ is a $\phi$-$n$-absorbing primary ideal of $R$, we may assume
that $I=I_{1}\times\cdots\times I_{n+1}\neq\psi_1(I_1)\times\psi_2(I_2)\times\cdots\times\psi_{n+1}(I_{n+1})$. So, there is an element
$(a_{1},\dots,a_{n+1})\in I\backslash(\psi_1(I_1)\times\psi_2(I_2)\times\cdots\times\psi_{n+1}(I_{n+1}))$. Then $(a_{1},1,\dots,1)(1,a_{2},1,\dots,1)\cdots(1,\dots,1,a_{n+1})\in I.$ Since $I$ is a $\phi$-$n$-absorbing primary ideal of $R$, then either
\begin{eqnarray*}
(a_{1},1,\dots,1)(1,a_{2},1,\dots,1)\cdots(1,\dots,1,a_{n},1)\\
=(a_{1},a_2,\dots,a_{n},1)\in I.
\end{eqnarray*}
or, for some $1\leq i\leq n$ we have
$$
(a_{1},1,\dots,1)\cdots(1,\dots,1,a_{i-1},1,\dots,1)(1,\dots,1,a_{i+1},1,\dots,1)\cdots$$
$$(1,\dots,1,a_{n+1})=(a_{1},\dots,a_{i-1},1,a_{i+1},\dots,a_{n+1})\in\sqrt{I}.
$$
Then $I_{i}=R_{i}$, for some $1\leq i\leq n+1$ and so $I=I_1\times\cdots I_{i-1}\times R_i\times I_{i+1}\times\cdots I_{n+1}$. If $I\subseteq\sqrt{\phi(I)}$, then $\psi_i(R_i)=R_i$ which is a contradiction. Therefore, by Theorem \ref{nil}, 
$I$ must be an $n$-absorbing primary ideal of $R$.
\end{proof}

\begin{theorem}
Let $R=R_{1}\times\cdots\times R_{n+1}$ where $R_{i}$'s are 
rings with identity and let  for $i= 1, 2,\dots,n+1$, $\psi_i : {\mathfrak{I}}(R_i)\to {\mathfrak{I}}(R_i)\cup \{\emptyset\}$ be a function such that $\psi_{i}(R_{i})\neq R_i$. Set $\phi=\psi_1\times\psi_2\times\cdots\times\psi_{n+1}$. Let $L=I_{1}\times\dots\times I_{n+1}$ be a proper ideal of $R$ with $L\neq\phi(L)$. The following conditions are equivalent:
\begin{enumerate}
\item $L=I_{1}\times\dots\times I_{n+1}$ is a  $\phi$-$n$-absorbing primary ideal of $R$;
\item $L=I_{1}\times\dots\times I_{n+1}$ is an $n$-absorbing primary ideal of $R$;
\item $L=I_{1}\times\dots\times I_{i-1}\times R_{i}\times I_{i+1}\times\cdots\times I_{n+1}$ for some $1\leq i\leq n+1$ such that for each $1\leq t\leq n+1$ different from $i$, $I_{t}$ is a primary ideal of $R_{t}$ or $L=I_{1}\times\dots\times I_{\alpha_{1}-1}\times R_{\alpha_{1}}\times I_{\alpha_{1}+1}\times\dots\times I_{\alpha_{j}-1}\times R_{\alpha_{j}}\times I_{\alpha_{j}+1}\cdots\times I_{n+1}$ in which $\{\alpha_{1},\dots,\alpha_{j}\}\subsetneq\{1,\dots,n+1\}$ and \vspace{-2mm}
$$\hspace{-6mm}I_{1}\times\dots\times I_{\alpha_{1}-1}\times I_{\alpha_{1}+1}\times\dots\times I_{\alpha_{j}-1}\times I_{\alpha_{j}+1}\cdots\times I_{n+1}$$ 
is an $n$-absorbing primary ideal of\vspace{-3mm} \[\hspace{5mm} R_{1}\times\dots\times R_{\alpha_{1}-1}\times R_{\alpha_{1}+1}\times\dots\times R_{\alpha_{j}-1}\times R_{\alpha_{j}+1}\times\cdots\times R_{n+1}.\]
\end{enumerate}
\end{theorem}
\begin{proof}
(1)$\Rightarrow$(2) Since $L$ is a $\phi$-$n$-absorbing primary ideal of $R$ and $L\neq\phi(L)$, then $L$ is an $n$-absorbing primary
ideal of $R$, by Theorem \ref{nonzero}.\\ 
(2)$\Rightarrow$(3) Suppose that $L$ is an $n$-absorbing primary ideal
of $R$, then for some $1\leq i\leq n+1$, $I_{i}=R_{i}$ by the proof
of Theorem \ref{nonzero}. Assume that $L=I_{1}\times\dots\times I_{i-1}\times R_{i}\times I_{i+1}\times\cdots\times I_{n+1}$ for $1\leq i\leq n+1$ such that for each $1\leq t\leq n+1$ different from $i$, $I_{t}$ is a proper ideal of $R_{t}$. Fix an $I_{t}$ different from $I_{i}$. We may assume that $t>i$. Let $ab\in I_{t}$ for some $a,b\in R_{t}$. In this case
{\small $$
\begin{array}{ll}
(0,1,\dots,1)(1,0,1,\dots,1)\cdots(1,\dots,1,0,\overbrace{1}^{i-th},\dots,1)(1,\dots,\overbrace{1}^{i-th},0,1,\dots,1)\cdots\\
(1,\dots,1,0,\overbrace{1}^{t-th},\dots,1)(1,\dots,\overbrace{1}^{t-th},0,1,\dots,1)\cdots(1,\dots,1,0)(1,\dots,1,\overbrace{a}^{t-th},1,\dots,1)\\
(1,\dots,1,\overbrace{b}^{t-th},1,\dots,1)=(0,\dots,0,\overbrace{1}^{i-th},0,\dots,0,\overbrace{ab}^{t-th},0,\dots,0)\in L.
\end{array}
$$}
Since $I_{1}\times\cdots\times I_{n+1}$ is $n$-absorbing primary and $I_{j}$'s different from $I_{i}$ are proper, then either
{\small $$
\begin{array}{ll}
(0,1,\dots,1)(1,0,1,\dots,1)\cdots(1,\dots,1,0,\overbrace{1}^{i-th},\dots,1)
(1,\dots,\overbrace{1}^{i-th},0,1,\dots,1)
\cdots\\(1,\dots,1,0,\overbrace{1}^{t-th},\dots,1)(1,\dots,\overbrace{1}^{t-th},0,1,\dots,1)
\cdots(1,\dots,1,0)(1,\dots,1,\overbrace{a}^{t-th},1,\dots,1)\\
\hspace{2.8cm}=(0,\dots,0,\overbrace{1}^{i-th},0,\dots,0,\overbrace{a}^{t-th},0,\dots,0)\in L,
\end{array}
$$}
or
{\small $$
\begin{array}{ll}
(0,1,\dots,1)(1,0,1,\dots,1)\cdots(1,\dots,1,0,\overbrace{1}^{i-th},\dots,1)
(1,\dots,\overbrace{1}^{i-th},0,1,\dots,1)\cdots\\
(1,\dots,1,0,\overbrace{1}^{t-th},\dots,1)(1,\dots,\overbrace{1}^{t-th},0,1,\dots,1)
\cdots(1,\dots,1,0)(1,\dots,1,\overbrace{b}^{t-th},1,\dots,1)\\
\hspace{2.8cm}=(0,\dots,0,\overbrace{1}^{i-th},0,\dots,0,\overbrace{b}^{t-th},0,\dots,0)\in\sqrt{L},
\end{array}
$$}
and thus either $a\in I_t$ or $b\in\sqrt{I_t}$. Consequently $I_t$ is a primary ideal of $R_t$.\\
Now, assume that $$L=I_{1}\times\dots\times I_{\alpha_{1}-1}\times R_{\alpha_{1}}\times I_{\alpha_{1}+1}\times\dots\times I_{\alpha_{j}-1}\times R_{\alpha_{j}}\times I_{\alpha_{j}+1}\times\cdots\times I_{n+1}$$ in which $\{\alpha_{1},\dots,\alpha_{j}\}\subsetneq\{1,\dots,n+1\}$.
Since $L$ is $n$-absorbing primary, then $I_{1}\times\dots\times I_{\alpha_{1}-1}\times I_{\alpha_{1}+1}\times\dots\times I_{\alpha_{j}-1}\times I_{\alpha_{j}+1}\cdots\times I_{n+1}$ is an $n$-absorbing primary ideal of $$R_{1}\times\dots\times R_{\alpha_{1}-1}\times R_{\alpha_{1}+1}\times\dots\times R_{\alpha_{j}-1}\times R_{\alpha_{j}+1}\times\cdots\times R_{n+1},$$ by Theorem \ref{T4}.\\
(3)$\Rightarrow$(1) If $L$ is in the first form, then similar to the proof of part (2)$\Rightarrow$(3) of Theorem \ref{prod} we can verify that
$L$ is an $n$-absorbing primary ideal of $R$, and hence $L$ is a $\phi$-$n$-absorbing primary ideal of $R$. For the second form apply Theorem \ref{T4}.
\end{proof}

\begin{theorem}\label{field}
Let $R=R_{1}\times\cdots\times R_{n+1}$ where $R_{i}$'s are 
rings with identity and let  for $i= 1, 2,\dots,n+1$, $\psi_i : {\mathfrak{I}}(R_i)\to {\mathfrak{I}}(R_i)\cup \{\emptyset\}$ be a function. Set $\phi=\psi_1\times\psi_2\times\cdots\times\psi_{n+1}$. If every proper ideal of $R$ is a $\phi$-$n$-absorbing primary ideal $(\phi$-$n$-absorbing ideal$)$ of $R$, then $I=\psi_i(I)$ for every $1\leq i\leq n+1$ and every proper ideal $I$ of $R_i$. The converse holds if in addition $\psi_i(R_i)=R_i$ for every $1\leq i\leq n+1$.
\end{theorem}
\begin{proof}
Fix an $i$ and let $I$ be a proper ideal of $R_i$. Assume that  $I\neq\psi_i(I)$, so give an element 
$x\in I\backslash\psi_i(I)$. Set $$J:=I\times\{0\}\cdots\times\{0\}.$$Notice that 
$$(1,0,1,\dots,1)(1,1,0,1,\dots,1)\cdots(1,\dots,1,0)(x,1,\dots,1)\in J\backslash\phi(J)$$
Since $I$ is $\phi$-$n$-absorbing primary, then either $$(1,0,1,\dots,1)(1,1,0,1,\dots,1)\cdots(1,\dots,1,0)\in J,$$ 
or the product of $(x,1,\dots,1)$ with $n-1$ of $(1,0,1,\dots,1),(1,1,0,1,\dots,1),\dots,\\(1,\dots,1,0)$ is in $\sqrt{J}$ which
implies that either $1\in I$ or $1\in\{0\}$, a contradiction. Consequently $I=\psi_i(I)$.
\end{proof}

\begin{corollary}
Let $n\geq 2$ be a natural number and $R=R_{1}\times\cdots\times R_{n+1}$ be a decomposable 
ring with identity. The following conditions are equivalent:
\begin{enumerate}
\item $R$ is a von Neumann regular ring;
\item Every proper ideal of $R$ is an $n$-almost $n$-absorbing primary ideal of $R$;
\item Every proper ideal of $R$ is an $\omega$-$n$-absorbing primary ideal of $R$;
\item Every proper ideal of $R$ is an $n$-almost $n$-absorbing ideal of $R$.
\end{enumerate}
\end{corollary}
\begin{proof}
(1)$\Leftrightarrow$(2), (1)$\Leftrightarrow$(3) and (1)$\Leftrightarrow$(4): Notice that, $\phi_{n}(I)=I$ (or $\phi_{\omega}(I)=I$)
if and only if $I=I^{2}$. By the fact that $R$ is von Neumann regular if and only if $I=I^2$ for every ideal $I$ of $R$ and regarding Theorem \ref{field} we have the implications.
\end{proof}

\begin{corollary}
Let $R_1,R_2,\dots,R_{n+1}$ be  rings and let $R = R_1\times R_2\times\cdots\times R_{n+1}$. Then the following conditions are equivalent:
 \begin{enumerate}
 \item $R_1,R_2,\dots,R_{n+1}$ are fields;
 \item Every proper ideal of $R$ is a weakly $n$-absorbing ideal of $R$;
 \item Every proper ideal of $R$ is a weakly $n$-absorbing primary ideal of $R$.
 \end{enumerate}
\end{corollary}
\begin{proof}
(1)$\Rightarrow$(2) By \textrm{\cite[Theorem 1.10]{EN}}.\\
(2)$\Rightarrow$(3) is clear.\\
(3)$\Rightarrow$(1) In Theorem \ref{field} assume that $\phi=0$.\\
\end{proof}

\section{The stability of $\phi$-$n$-absorbing primary ideals with respect to idealization}

Let $R$ be a commutative ring and $M$ be an $R$-module. We recall from  \cite[Theorem 25.1]{Huk} that every ideal of $R(+)M$ is in the form of 
$I(+)N$ in which $I$ is an ideal of $R$ and $N$ is a submodule of $M$ such that $IM\subseteq N$. Moreover, if $I_1(+)N_1$ and 
$I_2(+)N_2$ are ideals of $R(+)M$, then $(I_1(+)N_1)\cap(I_2(+)N_2)=(I_1\cap I_2)(+)(N_1\cap N_2)$.

\begin{theorem}\label{idealiz}
Let $R$ be a ring, $I$ a proper ideal of $R$ and $M$ an $R$-module. Suppose that $\psi: {\mathfrak{I}}(R)\to {\mathfrak{I}}(R)\cup \{\emptyset\}$ and $\phi:{\mathfrak{I}}(R(+)M)\to {\mathfrak{I}}(R(+)M)\cup \{\emptyset\}$ are two functions such that $\phi(I(+)M)=\psi(I)(+)N$ for some submodule $N$ of $M$
with $\psi(I)M\subseteq N$. Then the following conditions are equivalent:
\begin{enumerate}
\item $I(+)M$ is a $\phi$-$n$-absorbing primary ideal of $R(+)M$;
\item $I$ is a $\psi$-$n$-absorbing primary ideal of $R$ and if $(a_1,\dots,a_{n+1})$ is a $\psi$-$(n+1)$-tuple, then
 the second component of $(a_1,m_1)\cdots(a_{n+1},m_{n+1})$ is in $N$ for any elements $m_1,\dots,m_{n+1}\in M$.
\end{enumerate}
\end{theorem}
\begin{proof}
$(1)\Rightarrow(2)$ Assume that $I(+)M$ is a $\phi$-$n$-absorbing primary ideal of $R(+)M$. Let $x_1\cdots x_{n+1}\in I\backslash\psi(I)$
for some $x_1,\dots,x_{n+1}\in R$. Therefore $$(x_1,0)\cdots(x_{n+1},0)=(x_1\cdots x_{n+1},0)\in I(+)M\backslash\phi(I(+)M),$$
because $\phi(I(+)M)=\psi(I)(+)N$. Hence either $(x_1,0)\cdots(x_{n},0)=(x_1\cdots x_{n},0)\in I(+)M$
or $(x_1,0)\cdots\widehat{(x_i,0)}\cdots(x_{n+1},0)=(x_1\cdots\widehat{x_i}\cdots x_{n+1},0)\in\sqrt{I(+)M}=\sqrt{I}(+)M$ 
for some $1\leq i\leq n$. So either $x_1\cdots x_n\in I$ or $x_1\cdots\widehat{x_i}\cdots x_{n+1}\in\sqrt{I}$ for some $1\leq i\leq n$
which shows that $I$ is $\psi$-$n$-absorbing primary. For the second statement suppose that $a_1\cdots a_{n+1}\in\psi(I)$, $a_1\cdots a_n\notin I$ and $a_1\cdots\widehat{a_i}\cdots a_{n+1}\notin\sqrt{I}$ for all $1\leq i\leq n$. If the second component of $(a_1,m_1)\cdots(a_{n+1},m_{n+1})$ is not in $N$, then $$(a_1,m_1)\cdots(a_{n+1},m_{n+1})\in I(+)M\backslash\psi(I)(+)N,$$
Thus either $(a_1,m_1)\cdots(a_{n},m_{n})\in I(+)M$ or $(a_1,m_1)\cdots\widehat{(a_i,m_i)}\cdots(a_{n+1},m_{n+1})\in\sqrt{I}(+)M$ for some $1\leq i\leq n$. So either $a_1\cdots a_n\in I$ or $a_1\cdots\widehat{a_i}\cdots a_{n+1}\in\sqrt{I}$ for some $1\leq i\leq n$,
which is a contradiction.\\
$(2)\Rightarrow(1)$ Suppose that $(a_1,m_1)\cdots(a_{n+1},m_{n+1})\in I(+)M\backslash\psi(I)(+)N$ for some $a_1,\dots,a_{n+1}\in R$
and some $m_1,\dots,m_{n+1}\in M$. Clearly $a_1\cdots a_{n+1}\in I$. If $a_1\cdots a_{n+1}\in\psi(I)$, then the second component
of $(a_1,m_1)\cdots(a_{n+1},m_{n+1})$ cannot be in $N$. Hence either $a_1\cdots a_n\in I$ or $a_1\cdots\widehat{a_i}\cdots a_{n+1}\in\sqrt{I}$ for some $1\leq i\leq n$. If $a_1\cdots a_{n+1}\notin\psi(I)$, then $I~~~~\psi$-$n$-absorbing primary 
implies that either $a_1\cdots a_n\in I$ or $a_1\cdots\widehat{a_i}\cdots a_{n+1}\in\sqrt{I}$ for some $1\leq i\leq n$. 
Therefore we have either $(a_1,m_1)\cdots(a_{n},m_{n})\in I(+)M$ or $(a_1,m_1)\cdots\widehat{(a_i,m_i)}\cdots(a_{n+1},m_{n+1})\in\sqrt{I(+)M}$ for some $1\leq i\leq n$. Consequently 
$I(+)M$ is a $\phi$-$n$-absorbing primary ideal of $R(+)M$.
\end{proof}

\begin{corollary}
Let $R$ be a ring, $I$ be a proper ideal of $R$ and $M$ be an $R$-module. The following conditions are equivalent:
\begin{enumerate}
\item $I(+)M$ is an $n$-absorbing primary ideal of $R(+)M$;
\item $I$ is an $n$-absorbing primary ideal of $R$.
\end{enumerate}
\end{corollary}
\begin{proof}
In Theorem \ref{idealiz} set $\phi=\emptyset$, $\psi=\emptyset$ and $N=M$.
\end{proof}

\begin{corollary}
Let $R$ be a ring, $I$ be a proper ideal of $R$ and $M$ be an $R$-module. The following conditions are equivalent:
\begin{enumerate}
\item $I(+)M$ is a weakly $n$-absorbing primary ideal of $R(+)M$;
\item $I$ is a weakly $n$-absorbing primary ideal of $R$ and if $(a_1,\dots,a_{n+1})$ is an $(n+1)$-tuple-zero,
then the second component of $(a_1,m_1)\cdots(a_{n+1},m_{n+1})$ is zero for any elements $m_1,\dots, m_{n+1}\in M$.
\end{enumerate}
\end{corollary}
\begin{proof}
In Theorem \ref{idealiz} set $\phi=0$, $\psi=0$ and $N=\{0\}$.
\end{proof}

\begin{corollary}
Let $R$ be a ring, $I$ be a proper ideal of $R$ and $M$ be an $R$-module. Then the following conditions are equivalent:
\begin{enumerate}
\item $I(+)M$ is an $n$-almost $n$-absorbing primary ideal of $R(+)M$;
\item $I$ is an $n$-almost $n$-absorbing primary ideal of $R$ and if $(a_1,\dots,a_{n+1})$ is a $\phi_n$-$(n+1)$-tuple,
then for any elements $m_1,\dots,m_{n+1}\in M$ the second component of $(a_1,m_1)\cdots(a_{n+1},m_{n+1})$ is in $I^{n-1}M$.
\end{enumerate}
\end{corollary}
\begin{proof}
Notice that $(I(+)M)^n=I^n(+)I^{n-1}M$. In Theorem \ref{idealiz} set $\phi=\phi_n$, $\psi=\phi_n$ and $N=I^{n-1}M$.
\end{proof}

\begin{corollary}
Let $R$ be a ring, $I$ be a proper ideal of $R$ and $M$ be an $R$-module such that $IM=M$. Then $I(+)M$ is an $n$-almost $n$-absorbing primary ideal of $R(+)M$ if and only if $I$ is an $n$-almost $n$-absorbing primary ideal of $R$.
\end{corollary}

\begin{corollary}
Let $R$ be a ring, $I$ be a proper ideal of $R$ and $M$ be an $R$-module such that $IM=M$. Then $I(+)M$ is an $\omega$-$n$-absorbing primary ideal of $R(+)M$ if and only if $I$ is an $\omega$-$n$-absorbing primary ideal of $R$.
\end{corollary}

\section{Strongly $\phi$-$n$-absorbing primary ideals}

\begin{proposition}
Let $I$ be a proper ideal of a ring $R$. Then the following conditions are equivalent:
\begin{enumerate}
\item $I$ is strongly $\phi$-$n$-absorbing primary;
\item For every ideals $I_{1},\dots,I_{n+1}$ of $R$ such that
$I\subseteq I_{1}$, $I_{1}\cdots I_{n+1}\subseteq I\backslash\phi(I)$ implies that either 
$I_{1}\cdots I_{n}\subseteq I$ or $I_{1}\cdots\widehat{I_i}\cdots I_{n+1}\subseteq\sqrt{I}$
for some $1\leq i\leq n$.
\end{enumerate}
\end{proposition}

\begin{proof}
$(1)\Rightarrow(2)$ is clear.\\
$(2)\Rightarrow(1)$ Let $J,I_{2},\dots,I_{n+1}$ be ideals of $R$ such that
$JI_{2}\cdots I_{n+1}\subseteq I$ and $JI_{2}\cdots I_{n+1}\\\nsubseteq\phi(I)$. Then we
have that $(J+I)I_{2}\cdots I_{n+1}=(JI_{2}\cdots I_{n+1})+(II_{2}\cdots I_{n+1})\subseteq I$
and  $(J+I)I_{2}\cdots I_{n+1}\nsubseteq\phi(I)$.
Set $I_{1}:=J+I$. Then, by the hypothesis either $I_1\cdots I_{n}\subseteq I$ or $I_2\cdots I_{n+1}\subseteq\sqrt{I}$ or
there exists $2\leq i\leq n$ such that $(J+I)I_{2}\cdots\widehat{I_{i}}\cdots I_{n+1}\subseteq\sqrt{I}$. Therefore, either
$JI_{2}\cdots I_{n}\subseteq I$ or $I_2\cdots I_{n+1}\subseteq\sqrt{I}$ or
there exists $2\leq i\leq n$ such that $JI_{2}\cdots\widehat{I_{i}}\cdots I_{n+1}\subseteq \sqrt{I}$.
So $I$ is strongly $\phi$-$n$-absorbing primary.
\end{proof}

\begin{remark}
Let $R$ be a ring. Notice that ${\rm Jac}(R)$ is a radical ideal of $R$. So ${\rm Jac}(R)$ is a strongly $n$-absorbing ideal of $R$ if and only if $I$ is a strongly $n$-absorbing primary ideal of $R$.
\end{remark}

Given any set $X$, one can define a topology on $X$ where every
subset of $X$ is an open set. This topology is referred to as the discrete
topology on $X$, and $X$ is a discrete topological space if it is equipped with its discrete topology.\\
We denote by ${\rm Max}(R)$ the set of all maximal ideals of $R$.
\begin{theorem}
Let $R$ be a ring and $Max(R)$ be a discrete topological space. Then $Max(R)$ is an infinite set if and only if $Jac(R)$ is not strongly $n$-absorbing for every natural number $n$.
\end{theorem}
\begin{proof}
$(\Leftarrow)$ We can verify this implication without any assumption on ${\rm Max}(R)$, by \cite[Theorem 2.1]{AB1}.\\
$(\Rightarrow)$ Notice that ${\rm Max}(R)$ is a discrete topological space if and only if the Jacobson radical
of $R$ is the irredundant intersection of the maximal ideals of $R$, \cite[Corollary 3.3]{Ali}. Let ${\rm Max}(R)$ be an infinite set. Assume that for some natural number $n$, ${\rm Jac}(R)$ is a strongly $n$-absorbing ideal. Choose $n$ distinct elements $M_1,M_2,\dots,M_n$ of ${\rm Max}(R)$. Set $\mathcal{M}:=\{M_1,M_2,\dots,M_n\}$, and denote by $\mathcal{M}^c$ the complement of $\mathcal{M}$ in ${\rm Max}(R)$. Since 
${\rm Jac}(R)=M_1\cap M_2\cap\dots\cap M_n\cap(\bigcap\limits_{M\in{\mathcal{M}}^c}M)$, then either $M_1\cdots M_{i-1} M_{i+1}\cdots M_n(\bigcap\limits_{M\in{\mathcal{M}}^c}M)\subseteq {\rm Jac}(R)$ for some $1\leq i\leq n$, or
$M_1M_{2}\cdots M_n\subseteq {\rm Jac}(R)$.
In the first case we have $M_1\cdots M_{i-1} M_{i+1}\cdots M_n(\bigcap\limits_{M\in{\mathcal{M}}^c}M)\\\subseteq M_i$ and so $\bigcap\limits_{M\in{\mathcal{M}}^c}M\subseteq M_i$, a contradiction. If $M_1M_{2}\cdots M_n\subseteq {\rm Jac}(R)$,
then $M_1M_{2}\cdots M_n\subseteq M$
for every $M\in\mathcal{M}^c$, and so again we reach a contradiction. Consequently ${\rm Jac}(R)$ is not strongly $n$-absorbing.
\end{proof}

In the next theorem we investigate $\phi$-$n$-absorbing primary ideals over $u$-rings.
Notice that any B$\acute{\rm e}$zout ring is a $u$-ring, \cite[Corollary 1.2]{Q}.

\begin{theorem}\label{main}
Let $R$ be a $u$-ring and let $\phi:\mathfrak{J}(R)\rightarrow \mathfrak{J}(R)\cup\{\emptyset\}$
be a function. Then the following conditions are equivalent:
\begin{enumerate}
\item $I$ is strongly $\phi$-$n$-absorbing primary;
\item $I$ is $\phi$-$n$-absorbing primary;
\item For every elements $x_{1},\dots,x_{n}\in R$ with $x_{1}\cdots x_{n}\notin\sqrt{I}$ either 
$$(I:_{R}x_{1}\cdots x_{n})=(I:_{R}x_{1}\cdots x_{n-1})$$ or
$(I:_{R}x_{1}\cdots x_{n})\subseteq(\sqrt{I}:_{R}x_{1}\cdots\widehat{x_{i}}\cdots x_{n})$ for some $1\leq i\leq n-1$  or $(I:_{R}x_{1}\cdots x_{n})=(\phi(I):_{R}x_{1}\cdots x_{n});$
\item For every $t$ ideals $I_{1},\dots,I_{t}$, $1\leq t\leq n-1$, and for every elements $x_{1},\dots,x_{n-t}\\\in R$ such that $x_{1}\cdots x_{n-t}I_{1}\cdots I_{t}\nsubseteq\sqrt{I}$, 
$$(I:_{R}x_{1}\cdots x_{n-t}I_{1}\cdots I_{t})=(I:_{R}x_{1}\cdots x_{n-t-1}I_1\cdots I_{t})$$
or
$$(I:_{R}x_{1}\cdots x_{n-t}I_{1}\cdots I_{t})\subseteq(\sqrt{I}:_{R}x_{1}\cdots\widehat{x_{i}}\cdots x_{n-t}I_{1}\cdots I_{t})$$ for some $1\leq i\leq n-t-1$
or $$(I:_{R}x_{1}\cdots x_{n-t}I_{1}\cdots I_{t})\subseteq(\sqrt{I}:_{R}x_{1}\cdots x_{n-t}I_{1}\cdots\widehat{I_{j}}\cdots I_{t})$$ for some $1\leq j\leq t$ or 
$$(I:_{R}x_{1}\cdots x_{n-t}I_{1}\cdots I_{t})=(\phi(I):_{R}x_{1}\cdots x_{n-t}I_{1}\cdots I_{t}).$$
\item For every ideals $I_1,I_2,\dots,I_n$ of $R$ with $I_1I_2\cdots I_n\nsubseteq I$, either there is $1\leq i\leq n$ such that 
$(I:_RI_1\cdots I_n)\subseteq(\sqrt{I}:_RI_1\cdots\widehat{I_{i}}\cdots I_{n})$ or $(I:_RI_1\cdots I_n)=(\phi(I):_RI_1\cdots I_{n})$.
\end{enumerate}
\end{theorem}
\begin{proof}
(1)$\Rightarrow$(2) It is clear.\\
(2)$\Rightarrow$(3) Suppose that $x_{1},\dots,x_{n}\in R$ such that $x_{1}\cdots x_{n}\notin\sqrt{I}$. By Theorem \ref{main11}, $$(I:_{R}x_{1}\cdots x_{n})\subseteq[\cup_{i=1}^{n-1}(\sqrt{I}:_{R}x_{1}\cdots\widehat{x_{i}}\cdots x_{n})]\cup$$
$$\hspace{4.3cm}(I:_{R}x_{1}\cdots x_{n-1})\cup(\phi(I):_{R}x_{1}\cdots x_{n}).$$ 
Since $R$ is a $u$-ring we have either 
$(I:_{R}x_{1}\cdots x_{n})\subseteq(\sqrt{I}:_{R}x_{1}\cdots\widehat{x_{i}}\cdots x_{n})$ for some $1\leq i\leq n-1$  or
$(I:_{R}x_{1}\cdots x_{n})=(I:_{R}x_{1}\cdots x_{n-1})$ or $(I:_{R}x_{1}\cdots x_{n})=(\phi(I):_{R}x_{1}\cdots x_{n})$.\\
(3)$\Rightarrow$(4)
We use induction on $t$. For $t=1$, consider elements $x_1,\dots, x_{n-1}\in R$ and ideal $I_1$ of $R$ such that
$x_1\cdots x_{n-1}I_1\nsubseteq\sqrt{I}$. Let $a\in(I:_Rx_1\cdots x_{n-1}I_1)$. Then $I_1\subseteq(I:_Rax_1\cdots x_{n-1})$. If
$ax_1\cdots x_{n-1}\in\sqrt{I}$, then $a\in(\sqrt{I}:_Rx_1\cdots x_{n-1})$. If $ax_1\cdots x_{n-1}\notin\sqrt{I}$, then by part (3),
either $I_1\subseteq(I:_Rax_1\cdots x_{n-2})$ or $I_1\subseteq(\sqrt{I}:_Rax_1\cdots\widehat{x_i}\cdots x_{n-1})$
for some $1\leq i\leq n-2$ or $I_1\subseteq(\sqrt{I}:_Rx_1\cdots x_{n-1})$ or $I_1\subseteq(\phi(I):_Rax_1\cdots x_{n-1})$. 
The first case implies that $a\in(I:_Rx_1\cdots x_{n-2}I_1)$. The second case implies that 
$a\in(\sqrt{I}:_Rx_1\cdots\widehat{x_i}\cdots x_{n-1}I_1)$ for some $1\leq i\leq n-2$. The third case cannot be happen, because 
$x_1\cdots x_{n-1}I_1\nsubseteq\sqrt{I}$, and the last case implies that $a\in(\phi(I):_Rx_1\cdots x_{n-1}I_1)$. Hence  
$$(I:_{R}x_{1}\cdots x_{n-1}I_1)\subseteq\cup_{i=1}^{n-2}(\sqrt{I}:_{R}x_{1}\cdots\widehat{x_{i}}\cdots x_{n-1}I_1)\cup(\sqrt{I}:_Rx_1\cdots x_{n-1})$$$$\hspace{2.6cm}\cup(I:_Rx_1\cdots x_{n-2}I_1)\cup(\phi(I):_{R}x_{1}\cdots x_{n-1}I_1).$$ Since $R$ is a $u$-ring, then either
$(I:_{R}x_{1}\cdots x_{n-1}I_1)\subseteq(\sqrt{I}:_{R}x_{1}\cdots\widehat{x_{i}}\cdots x_{n-1}I_1)$
for some $1\leq i\leq n-2$, or $(I:_{R}x_{1}\cdots x_{n-1}I_1)\subseteq(\sqrt{I}:_Rx_1\cdots x_{n-1})$
or $(I:_{R}x_{1}\cdots x_{n-1}I_1)=(I:_Rx_1\cdots x_{n-2}I_1)$ or 
$(I:_{R}x_{1}\cdots x_{n-1}I_1)=(\phi(I):_{R}x_{1}\cdots x_{n-1}I_1).$ Now suppose $t>1$ and assume that for integer $t-1$ the claim holds. Let $x_{1},\dots,x_{n-t}$ be elements of $R$ and let $I_{1},\dots,I_{t}$ be ideals of $R$ such that \\ $x_{1}\cdots x_{n-t}I_{1}\cdots I_{t}\nsubseteq\sqrt{I}$. Consider element $a\in(I:_{R}x_{1}\cdots x_{n-t}I_{1}\cdots I_{t})$. Thus $I_{t}\subseteq(I:_{R}ax_{1}\cdots x_{n-t}I_{1}\cdots I_{t-1})$. If $ax_{1}\cdots x_{n-t}I_{1}\cdots I_{t-1}\subseteq\sqrt{I}$, then $a\in(\sqrt{I}:_Rx_{1}\cdots x_{n-t}I_{1}\cdots I_{t-1})$. If $ax_{1}\cdots x_{n-t}I_{1}\cdots I_{t-1}\nsubseteq\sqrt{I}$, then by induction hypothesis, either
$$(I:_{R}ax_{1}\cdots x_{n-t}I_{1}\cdots I_{t-1})\subseteq(\sqrt{I}:_{R}x_{1}\cdots x_{n-t}I_{1}\cdots I_{t-1})$$ 
or
 $$(I:_{R}ax_{1}\cdots x_{n-t}I_{1}\cdots I_{t-1})\subseteq(\sqrt{I}:_{R}ax_{1}\cdots\widehat{x_{i}}\cdots x_{n-t}I_{1}\cdots I_{t-1})$$ for some $1\leq i\leq n-t-1$ 
or
$$(I:_{R}ax_{1}\cdots x_{n-t}I_{1}\cdots I_{t-1})\subseteq(\sqrt{I}:_{R}ax_{1}\cdots x_{n-t}I_{1}\cdots\widehat{I_{j}}\cdots I_{t-1})$$ for some $1\leq j\leq t-1$
 or 
$$(I:_{R}ax_{1}\cdots x_{n-t}I_{1}\cdots I_{t-1})=(I:_{R}ax_{1}\cdots x_{n-t-1}I_{1}\cdots I_{t-1}),$$  
or
$(I:_{R}ax_{1}\cdots x_{n-t}I_{1}\cdots I_{t-1})=(\phi(I):_{R}ax_{1}\cdots x_{n-t}I_{1}\cdots I_{t-1})$. Since $x_{1}\cdots x_{n-t}\\I_{1}\cdots I_{t}\nsubseteq\sqrt{I}$, then the first case cannot happen. Consequently, either $a\in(\sqrt{I}:_{R}x_{1}\cdots\widehat{x_{i}}\cdots x_{n-t}I_{1}\cdots I_{t})$ 
for some $1\leq i\leq n-t-1$ or
$a\in(\sqrt{I}:_{R}x_{1}\cdots x_{n-t}I_{1}\cdots\widehat{I_{j}}\cdots I_{t})$ for some $1\leq j\leq t-1$
 or 
$a\in(I:_{R}x_{1}\cdots x_{n-t-1}I_{1}\cdots I_{t}),$ 
or
$a\in(\phi(I):_{R}x_{1}\cdots x_{n-t}I_{1}\cdots I_{t})$.
Hence 
\begin{eqnarray*}
(I:_{R}x_{1}\cdots x_{n-t}I_{1}\cdots I_{t})&\subseteq&[\cup_{i=1}^{n-t-1}(\sqrt{I}:_{R}x_{1}\cdots\widehat{x_{i}}\cdots x_{n-t}I_{1}\cdots I_{t})]\\
&\cup&[\cup_{j=1}^{t}(\sqrt{I}:_{R}x_{1}\cdots x_{n-t}I_{1}\cdots\widehat{I_{j}}\cdots I_{t})]\\
&\cup&(I:_{R}x_{1}\cdots x_{n-t-1}I_{1}\cdots I_{t})\\
&\cup&(\phi(I):_{R}x_{1}\cdots x_{n-t}I_{1}\cdots I_{t}).
\end{eqnarray*}
Now, since $R$ is $u$-ring we are done.\\ 
(4)$\Rightarrow$(5) Let $I_1,I_2,\dots,I_n$ be ideals of $R$ such that $I_1I_2\cdots I_n\nsubseteq I$. Suppose that
$a\in(I:_RI_1I_2\cdots I_n)$. Then $I_n\subseteq(I:_RaI_1I_2\cdots I_{n-1})$. If $aI_1I_2\cdots I_{n-1}\subseteq\sqrt{I}$,
then $a\in(\sqrt{I}:_RI_1I_2\cdots I_{n-1})$. If $aI_1I_2\cdots I_{n-1}\nsubseteq\sqrt{I}$, then by part (4) we have either
$I_n\subseteq(I:_RI_1I_2\cdots I_{n-1})$ or $I_n\subseteq(\sqrt{I}:_RaI_1\cdots\widehat{I_i}\cdots I_{n-1})$ for some $1\leq i\leq n-1$
or $I_n\subseteq(\phi(I):_RaI_1I_2\cdots I_{n-1})$. By hypothesis, the first case is not hold. The second case imples that $a\in(\sqrt{I}:_RI_1\cdots\widehat{I_i}\cdots I_{n})$
for some $1\leq i\leq n-1$. The third case implies that $a\in(\phi(I):_RI_1I_2\cdots I_{n})$. Similarly, since $R$ is $u$-ring, there is $1\leq i\leq n$ such that $(I:_RI_1\cdots I_n)\subseteq(\sqrt{I}:_RI_1\cdots\widehat{I_{i}}\cdots I_{n})$ or $(I:_RI_1\cdots I_n)=(\phi(I):_RI_1\cdots I_{n})$.\\
 (5)$\Rightarrow$(1) This implication has an easy verification.
\end{proof}

\begin{remark}
Note that in Theorem \ref{main}, for the case $n=2$ and $\phi=\phi_{\emptyset}$ we can omit the condition $u$-ring, 
by the fact that if an ideal (a subgroup) is the union of two ideals (two
subgroups), then it is equal to one of them. So we conclude that an ideal $I$ of a ring $R$ is 2-absorbing primary if and only if
it is strongly 2-absorbing primary.
\end{remark}
Let $R$ be a  ring with identity. We recall that if $f=a_0+a_1X+\cdots+a_tX^t$
is a polynomial on the ring $R$, then {\it content} of $f$ is defined as the $R$-ideal, generated by the coefficients of $f$, i.e.
$c( f )=(a_0,a_1,\dots,a_t)$. Let $T$ be an $R$-algebra and $c$ the function from $T$ to the ideals of $R$ defined by  $c(f)=\cap\{I\mid$ $I$ \mbox{is an ideal of} $R$ \mbox{and} $f\in IT\}$ known as the content of $f$. Note that the content function $c$ is nothing but the generalization of the content of a polynomial $f\in R[X]$. The $R$-algebra $T$ is called a {\it content $R$-algebra} if the following conditions hold:
\begin{enumerate}
\item For all $f\in T$, $f\in c(f)T$.
\item (Faithful flatness ) For any $r\in R$ and $f\in T$, the equation $c(rf )=rc(f)$ holds and
$c(1_T)=R$.
\item (Dedekind-Mertens content formula) For each $f,g\in T$, there exists a natural
number $n$ such that $c(f)^nc(g)=c(f)^{n-1}c(fg)$.
\end{enumerate}
For more information on content algebras and their examples we refer to \cite{north}, \cite{ohm} and \cite{rush}.
In \cite{nas} Nasehpour gave the definition of a Gaussian $R$-algebra as follows: Let $T$ be an $R$-algebra such that $f\in c(f)T$ for all $f\in T$. $T$ is said to be a Gaussian $R$-algebra if $c(fg)=c( f )c(g)$, for all $f,g\in T$.
\begin{example}(\cite{nas})
Let $T$ be a content $R$-algebra such that $R$ is a Pr\"{u}fer domain. Since every
nonzero finitely generated ideal of $R$ is a cancellation ideal of $R$, the Dedekind-Mertens 
content formula causes $T$ to be a Gaussian $R$-algebra.
\end{example}

In the following theorem we use the functions $\phi_R$ and $\phi_T$ that defined just prior to Theorem \ref{context}.
\begin{theorem}
Let R be a Pr\"{u}fer domain, $T$ a content $R$-algebra and $I$ an ideal of $R$. Then $I$ is a $\phi_R$-$n$-absorbing primary ideal of $R$ if and only if $IT$ is a $\phi_T$-$n$-absorbing primary ideal of $T$.
\end{theorem}
\begin{proof}
Assume that $I$ is a $\phi_R$-$n$-absorbing primary ideal of $R$. Let $f_1f_2\cdots f_{n+1}\in IT\backslash\phi_T(IT)$ for some $f_1,f_2,\dots,f_{n+1}\in T$ such that $f_1f_2\cdots f_n\notin IT$. Then \\$c(f_1f_2\cdots f_{n+1})\subseteq I$. Since $R$ is a Pr\"{u}fer domain and $T$ is a content $R$-algebra, 
then $T$ is a Gaussian $R$-algebra. Therefore $c(f_1f_2\cdots f_{n+1})=c(f_1)c(f_2)\cdots c(f_{n+1})\\\subseteq I$. If 
$c(f_1f_2\cdots f_{n+1})\subseteq\phi_R(I)=\phi_T(IT)\cap R$, then $$f_1f_2\cdots f_{n+1}\in c(f_1f_2\cdots f_{n+1})T\subseteq(\phi_T(IT)\cap R)T\subseteq \phi_T(IT),$$ which is a contradiction. Hence $c(f_1)c(f_2)\cdots c(f_{n+1})\subseteq I$ and $c(f_1)c(f_2)\cdots c(f_{n+1})\\\nsubseteq\phi_R(I)$. Since $R$ is a $u$-domain, $I$ is a strongly $\phi_R$-$n$-absorbing primary ideal of $R$, by Theorem \ref{main}, and
this implies either $c(f_1)c(f_2)\cdots c(f_n)\subseteq I$ or $c(f_1)\cdots\widehat{c(f_i)}\cdots \\c(f_{n+1})\subseteq\sqrt{I}$ for some 
$1\leq i\leq n$. In the first case we have $f_1f_2\cdots f_n\in c(f_1f_2\cdots f_n)T\subseteq IT$, which contradicts our hypothesis. In the second case we have  $f_1\cdots\widehat{f_i}\cdots f_{n+1}\in(\sqrt{I})T\subseteq\sqrt{IT}$  for some 
$1\leq i\leq n$. Consequently $IT$ is a $\phi_T$-$n$-absorbing primary ideal of $T$.\\
For the converse, note that since $T$ is a content $R$-algebra, $IT\cap R=I$ for every ideal $I$ of $R$. Now, apply Theorem \ref{context}.
\end{proof}

The algebra of all polynomials over an arbitrary ring with an arbitrary number of indeterminates is an example of content algebras.
\begin{corollary}\label{last}
Let $R$ be a Pr\"{u}fer domain and $I$ be an ideal of $R$. Then $I$ is a $\phi_R$-$n$-absorbing primary ideal of $R$ if and only if $I[X]$ is a  $\phi_{R[X]}$-$n$-absorbing primary ideal of $R[X]$.
\end{corollary}

As two special cases of Corollary \ref{last}, when $\phi_R=\phi_T=\emptyset$ and $\phi_R=\phi_T=0$ we have the following result.
\begin{corollary}
Let $R$ be a Pr\"{u}fer domain and $I$ be an ideal of $R$. Then $I$ is an $n$-absorbing primary ideal of $R$ if and only if $I[X]$ is an $n$-absorbing primary ideal of $R[X]$.
\end{corollary}


\end{document}